\documentclass[12pt,reqno] {amsart}
\usepackage{amsfonts}
\usepackage{amsmath}
\usepackage{bbm}
\usepackage{amssymb}
\usepackage{xypic,color}
\usepackage{graphicx}%
\usepackage{hyperref}
\usepackage{cleveref}
\usepackage{comment}
 \usepackage{mathrsfs}
 \usepackage{float}
 \usepackage{tikz}
 \usepackage{bm}
 \usepackage{ulem}
\theoremstyle{plain}

\newtheorem{question}{Question}

\newtheorem{remark}{Remark}[section]
\newtheorem{example}{Example}[section]

\newtheorem{definition}{Definition}[section]
\numberwithin{equation}{section}
\newtheorem{theorem}{Theorem}[section]
\newtheorem{corollary}[theorem]{Corollary}
\newtheorem{lemma}[theorem]{Lemma}
\newtheorem{proposition}[theorem]{Proposition}
\newtheorem*{ac}{Acknowledgement}

\newtheorem{main theorem}{Main Theorem}

\newcommand{\fF}{\mathfrak{F}}

\newcommand{\sP}{\mathscr{P}}
\newcommand{\sC}{\mathscr{C}}

\newcommand{\FPdim}{\mathrm{FPdim}}

\newcommand{\Irr}{\mathrm{Irr}}
\newcommand{\Mu}{\mathbf{m}}

\begin{document}
\title[Complete Positivity and Categorification Criteria]{Complete Positivity of Comultiplication and Primary Criteria for Unitary Categorification}

\author{Linzhe Huang}
\address{L. Huang, Yau Mathematical Sciences Center, Tsinghua University, Beijing, 100084, China}
\email{huanglinzhe@mail.tsinghua.edu.cn}

\author{Zhengwei Liu}
\address{Z. Liu, Yau Mathematical Sciences Center and Department of Mathematics, Tsinghua University, Beijing, 100084, China}
\address{Yanqi Lake Beijing Institute of Mathematical Sciences and Applications, Huairou District, Beijing, 101408, China}
\email{liuzhengwei@mail.tsinghua.edu.cn}

\author{Sebastien Palcoux}
\address{S. Palcoux, Yanqi Lake  Beijing Institute of Mathematical Sciences and Applications, Beijing, 101408, China}
\email{sebastien.palcoux@gmail.com}

\author{Jinsong Wu}
\address{J. Wu, Yanqi Lake  Beijing Institute of Mathematical Sciences and Applications, Beijing, 101408, China}
\email{wjs@bimsa.cn}

\maketitle

\begin{abstract}
In this paper, we investigate quantum Fourier analysis on subfactors and unitary fusion categories. We prove the complete positivity of the comultiplication for subfactors and derive a primary $n$-criterion of unitary categorifcation of multifusion rings. It is stronger than the Schur product criterion when $n\geq3$. The primary criterion could be transformed into various criteria which are easier to check in practice even for noncommutative, high-rank, high-multiplicity, multifusion rings. 
More importantly, the primary criterion could be localized on a sparse set, so that it works for multifusion rings with sparse known data. 
We give numerous examples to illustrate the efficiency and the power of these criteria.

\end{abstract}



\tableofcontents

\section{Introduction}
Jones discovered a quantum knot invariant, the Jones polynomial \cite{Jon85}, inspired by the classification of Jones index of subfactors \cite{Jon83}. Witten extended the Jones polynomial as a three manifold invariant through certain partition function in
  topological quantum field theory in \cite{Wit88}.
Atiyah \cite{Ati88} introduced a mathematical axiomatization for topological quantum field theories, which computes the topological invariants for quantum field theories. 
One can construct 2+1 topological quantum field theory from modular tensor categories \cite{ResTur91} and from spherical categories \cite{TurVir92}. 
In addition, the partition function has reflection positivity, if the tensor categories are unitary which naturally come from subfactors \cite{Ocn94}. 

We study the analytic aspects of reflection positivity and the Fourier duality of subfactors in quantum Fourier analysis \cite{JJLRW20,JLW16,BJL17,JLW18,JLW19,LW18,HLW21}, which has led to various applications in subfactor theory
\cite{LPW21,Liu16,LMP15,BJL17,BDLR19}.
The positivity of convolution for subfactors, called the quantum Schur product theorem (Theorem 4.1 in \cite{Liu16}), turned out to be a surprisingly efficient analytic obstruction of unitary categorification of fusion rings, called the Schur product criterion (Proposition 8.3 and Corollary 8.5 in  \cite{LPW21}). This adds an extra dimension to algebraic obstructions for categorification that have been found in  \cite{EGNO15, Ost033, Ost15, LPR22}.

 In this paper, we show that the complete positivity of comultiplication is more fundamental than the positivity of convolution, which encodes stronger analytic obstructions of unitary categorification. In a companion paper \cite{HLW21}, we prove various convolution inequalities based on the positivity of comultiplication, which provide analytic obstructions as well.
Suppose $\mathcal{R}$ is a fusion ring \cite{Lus87} with fusion coefficients $N_{i,j}^k$,  $1\leq i,j,k\leq m$, and fusion matrices $M_i=N_{i,\cdot}^{\cdot}$.
We prove in Theorem \ref{thm:Higher Positivity Criterion} that if $\mathcal{R}$ has a unitary categorification, then
\begin{align}\label{eq:n quon positivity}
\sum_{i=1}^m \|M_i\|^2\bigg(\frac{M_i}{\|M_i\|}\bigg)^{\otimes n}\geq 0.
\end{align}
We call this inequality the \textbf{primary $n$-criterion}.
The primary $n$-criterion is stronger than the Schur product criterion when $n\geq3$, more friendly to noncommutative fusion rings, and easier to check by computer in practice. 

More importantly, for any subset $S\subseteq \{1,2,\ldots,m\}$, denote $M_i^S$ to be the submatrix of $M_i$,  we can derive the following localized categorification criteria in Theorem \ref{thm:localization criterion} from Inequality \eqref{eq:n quon positivity},
\begin{align}\label{eq:n positivity locolization criterion intro}
\sum_{i=1}^m \|M_i\|^2\bigg(\frac{M_i^S}{\|M_i\|}\bigg)^{\otimes n}\geq 0.
\end{align}
By Frobenius reciprocity, $N_{i,j}^k=N_{j,k^*}^{i^*}$, which only depends on the local fusion rules on $S$.
 We call Inequality \eqref{eq:n positivity locolization criterion intro} the \textbf{$S$-localized $n$-criterion}. 
The size of the checking matrices in the primary 
criteria is large when the rank of the fusion ring is large. 
It would take a long time for computers to check the positivity. Fortunately,  the submatrices inherit the positivity of matrices.
Based on this fact, the localized criteria make it more efficient to check the positivity when the size of the local set $S$ is fixed. Moreover, the localized criteria are still valid  when only sparse data of
the fusion rules are known.

Such local obstructions are rare.  The triple/quadruple point obstructions  \cite{Haa94,Jon12, MPPS12, Sny13, IJMS12, Pen15} are local obstructions to eliminate certain bipartite graphs with a triple/quadruple point from principal graph of subfactors, which are crucial in the small-index classification of subfactors \cite{JMS14}. 
Technically, a subfactor could produce a ``fusion ring" with 2-color. Inspired by this, we generalize our results to multifusion rings in \S \ref{sec:multifusion ring and principal graph}, including a general version of the localized categorification criterion  in Theorem \ref{thm:local criteria multifusion ring}.
Applying this criterion, we eliminate a large family of bipartite graphs of the following form from principal graphs of subfactor  with certain dimension bounds (See Theorem \ref{thm:infinite principle graph exclude} and Remark \ref{rem:infinite principle graph exclude}).
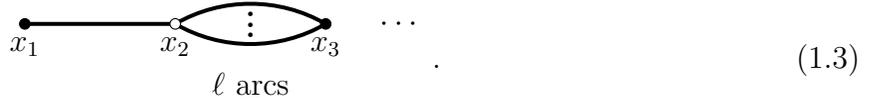
\begin{align}\label{fig:principle graph 1}
\raisebox{-.5cm}{
\begin{tikzpicture}
\draw[line width=1.5](0,0) --(2,0);
\draw(0,-0.3)node{$x_1$};
\draw(2,-0.3)node{$x_2$};
\draw(4,-0.3)node{$x_3$};
\draw (2,0)[line width=1.5] arc (120:60:2);
\draw (2,0)[line width=1.5] arc (-120:-60:2);
\draw(3,0.1)node{\textbf{$\vdots$}};
\draw(3,-0.8)node{$\ell$ arcs};
\draw(5,0)node{\textbf{$\cdots$}};
\draw [fill=black] (0, 0) circle (0.07);
\draw [fill=white] (2, 0) circle (0.07);
\draw [fill=black] (4, 0) circle (0.07);
\end{tikzpicture}}.
\end{align}
It is the first local obstruction for principal graphs with arbitrary high multiplicities.

When $n$ grows, the computational complexity of checking positivity in Inequality \eqref{eq:n quon positivity} grows exponentially $O(m^{3n})$. We applied the Hadamard product to reduce the computation complexity to $O(m^3+m^2 \log n)$. 
For any subset $S\subseteq \{1,2,\ldots,m\}$ and unitary matrices $U_j\in M_{|S|}(\mathbb{C})$, $j=1,\ldots,n$,
 we can derive \textbf{reduced twisted $n$-criterion} in Theorem \ref{thm:hadamard criterion} from Inequality \eqref{eq:n positivity locolization criterion intro},
\begin{align}
\sum_{i=1}^m \|M_i\|^2\bigg(\frac{U_j M_i^S U_j^*}{\|M_i\|}\bigg)^{\ast_{j=1}^n}\geq 0,
\end{align}
where $\ast$ is the Hadamard product and $*_{j=1}^n$ means the Hadamard product of the $n$ matrices.
This twisted version is better in practice whenever there are enough zero entries of fusion matrices. We could choose proper unitary matrices to make it  computable even some entries of fusion matrices are unknown (See Theorem \ref{thm:two special simple object fusion ring} for application). 
Applying these criteria to fusion rings, we eliminate numerous fusing rings with sparse data from unitary categorification. We present the corresponding results as follows.

Fusion rings always pass primary $1,2$-criteria.
 By checking the dataset of \cite{VS22} with computer assistance, among the 28451 fusion rings, 19738 ones do not pass primary $3$-criterion (about 68.37\%). In particular, the primary $3$-criterion is more efficient for simple fusion rings.
Among 13893 simple fusion rings in \cite{VS22}, exactly 11729 ones can be excluded from unitary categorification by
 primary $3$-criterion (about $84.4\%$). One of the three simple noncommutative fusion rings  in the dataset of \cite{VS22}  can be excluded by primary $3$-criterion. We expect primary criteria to be powerful for fusion rings of high ranks.

Localized criteria allow us to eliminate fusion rings and bipartite graphs in a single time (See Subsection \ref{sub:S8} and Theorem \ref{thm:infinite principle graph exclude}). $\mathcal{K}_7$ is excluded by only  two points local set in \S \ref{sub:S7}. Other 9 fusion rings  are also excluded with a small perturbation of the Frobenius-Perron dimensions in  \S \ref{sec:fpdim}. 
 By applying the reduced criteria to a simple integral fusion ring of rank 8, we see that localized $n$-criteria are locally inequivalent for different $n$ (See \S \ref{sec:locineq}).  A family of infinitely many  simple fusion rings with unbounded dimension are excluded by
 reduced twisted criterion in Theorem \ref{thm:R4k}.

\begin{ac}
Linzhe Huang was supported by YMSC, Tsinghua University.
Zhengwei Liu was supported by NKPs (Grant no. 2020YFA0713000), by Tsinghua University (Grant no. 04200100122) and by Templeton Religion Trust (TRT 159).
Sebastien Palcoux was supported by grants from Yanqi Lake Beijing Institute of Mathematical Sciences and Applications.
Jinsong Wu was supported by NSFC
(Grant no. 12031004) and a grant from Yanqi Lake Beijing Institute of Mathematical Sciences and Applications.
\end{ac}

Data availability: The datasets generated during and/or analysed during the current study are available from the corresponding author on reasonable request.

\section{Preliminaries}\label{sec:Preliminaries}
In this section, we will recall the notion of fusion ring, fusion bi-algebra, subfactor and the connections between them, see \cite{LPW21} for further details.

Let $\mathbb{N}$ be the set of natural numbers, i.e. $\mathbb{N}=\{0, 1, 2, \ldots\}$. Let $\mathbb{R}_{\geq0}$ be the set of non-negative real numbers. Let $\mathbb{C}$ be the complex field.

A based ring $\mathcal{R}$ is a free $\mathbb{Z}$-module with basis $\{x_j\}_{j\in I}$ such that 
\begin{enumerate}
\item $\displaystyle x_ix_j=\sum_{k\in I} N_{ij}^k x_k$, where $N_{ij}^k\in \mathbb{N}$ and the sum here is a finite sum.
\item The identity $1$ is an $\mathbb{N}$-linear combination of the basis elements.
\item There exists an involution $*$ on $I$ such that the induced map 
\begin{align*}
\sum_{j\in I} n_j x_j \mapsto \sum_{j\in I} n_j x_{j^*}
\end{align*}
is an anti-involution of the ring $\mathcal{R}$.
\item The linear functional $\tau:\mathcal{R}\to \mathbb{Z}$ given by $\tau(x_j)=1$ if $x_j$ occurs in the decomposition of $1$ and $0$ otherwise satisfies that
\begin{align*}
\tau(x_i x_j)=\delta_{i, j^*}, \quad i, j\in I.
\end{align*}
\end{enumerate}
A unital based ring is a based ring such that the identity $1$ is a basis element.
A multifusion ring is a based ring of finite rank.
A fusion ring is a unital based ring of finite rank.

Furthermore, $\mathcal{R}$ is called a fusion algebra, if the condition $N_{ij}^k\in \mathbb{N}$ is released as $N_{ij}^k\geq 0$ 
\cite[Definition 2.1]{LPW21}.
We obtain a finite dimensional unital $C^*$-algebra $\mathcal{B}:=\mathcal{R}\otimes_\mathbb{Z}\mathbb{C}$ equipped with a faithful trace $\tau$.
In this paper, we identify $\mathcal{R}$ as $\mathcal{B}$ if there is no confusion.

We obtain another abelian $C^*$-algebra $\mathcal{A}$ with the basis $\{x_j\}_{j\in I}$ of a fusion ring $\mathcal{R}$, a multiplication $\diamond$ and an involution
$\#$, 
\begin{align}
   \begin{split}
x_i\diamond x_j&=\delta_{ij}d(x_i)^{-1}x_i,\\
x_i^{\#}&=x_i,
   \end{split}
\end{align}
where $d(x_j)$ is the operator norm of $x_j$ in $\mathcal{B}$, also called the Perron-Frobenius dimension of $x_j$. 
Then $d$ extends to a faithful trace on $\mathcal{A}$. The Fourier transform $\mathfrak{F}$: $\mathcal{A}\to\mathcal{B}$ is defined as
\begin{align*}
\mathfrak{F}(x_i)=x_i,\quad \forall i\in I.
\end{align*}
Then $\mathfrak{F}$ is a unitary from $L^2(\mathcal{A},d)$ to $L^2(\mathcal{B},\tau)$. The quintuple $(\mathcal{A},\mathcal{B},d,\tau,\mathfrak{F})$ is called the fusion bialgebra arising from $\mathcal{R}$. 
Given a fusion bialgebra $(\mathcal{A},\mathcal{B},d,\tau,\mathfrak{F})$, 
$(\mathcal{A},\mathcal{B},\lambda_1 d,\lambda_2\tau,\lambda_1^\frac{1}{2}\lambda_2^{-\frac{1}{2}}\mathfrak{F})$ is also a fusion bialgebra, for any gauge parameters $\lambda_1,\lambda_2>0$. 
The multiplication $\diamond$ on $\mathcal{A}$ induces a convolution on $\mathcal{B}$:
\begin{align}\label{def:convolution}
x_i\ast x_j:=\mathfrak{F}(x_i\diamond x_j)=\delta_{ij}d(x_i)^{-1}x_i.
\end{align}
The study of the positivity of the convolution led to the Schur product criterion \cite[Proposition 8.3]{LPW21} for unitary categorification of fusion rings. 

Jones introduced planar algebras \cite{Jon99} to study the standard invariants of subfactors. Suppose $\mathscr{P}$ is a subfactor planar algebra with Jones index $\delta^2$, and $\mathscr{P}_{2,+}$ is abelian, then the quintuple $(\mathscr{P}_{2,+},\mathscr{P}_{2,-},tr_{2,+},tr_{2,-},\mathfrak{F}_s)$ is a fusion bialgebra. Here $tr_{2,\pm}$ are the Markov traces on $\mathscr{P}_{2,\pm}$:
\begin{align}\label{eq:trace}
tr_{2,\pm}(x)=\raisebox{-0.9cm}{
\begin{tikzpicture}[scale=1.5]
\path [fill=white]   (0.2, 0.5) .. controls +(0, 0.375) and +(0, 0.375) .. (0.8, 0.5)--(0.8, 0).. controls +(0, -0.375) and +(0, -0.375) ..(0.2, 0);
\path [fill=white] (0.35, 0.5) .. controls +(0, 0.2) and +(0, 0.2) .. (0.65, 0.5)--(0.65, 0).. controls +(0, -0.2) and +(0, -0.2) ..(0.35, 0);
\draw [blue, fill=white] (0,0) rectangle (0.5, 0.5);
\node at (0.25, 0.25) {$x$};
\draw (0.35, 0.5) .. controls +(0, 0.2) and +(0, 0.2) .. (0.65, 0.5)--(0.65, 0).. controls +(0, -0.2) and +(0, -0.2) ..(0.35, 0);
\draw (0.15, 0.5) .. controls +(0, 0.45) and +(0, 0.45) .. (0.85, 0.5)--(0.85, 0).. controls +(0, -0.45) and +(0, -0.45) ..(0.15, 0);
\end{tikzpicture}}
\end{align}
Moreover, the multiplication and the convolution on  $\mathscr{P}_{2,\pm}$ are defined as
\begin{align*}
x y=\raisebox{-1.25cm}{
\begin{tikzpicture}[scale=1.5]
\draw (0.35, -.2) --++(0,1.6);
\draw (0.15, -.2) --++(0,1.6);
\path [fill=white] (0.35, 0.5) .. controls +(0, 0.3) and +(0, 0.3) .. (0.85, 0.5)-- (0.85, 0) .. controls +(0, -0.3) and +(0, -0.3) .. (0.35, 0)--(0.35, 0.5);
\draw [blue, fill=white] (0,0) rectangle (0.5, 0.5);
\node at (0.25, 0.25) {$x$};
\begin{scope}[shift={(0,0.7)}]
\draw [blue, fill=white] (0,0) rectangle (0.5, 0.5);
\node at (0.25, 0.25) {$y$};
\end{scope}
\end{tikzpicture}}
\;,\quad
x\ast y=\raisebox{-0.9cm}{
\begin{tikzpicture}[scale=1.5]
\path [fill=white] (0.35, 0.5) .. controls +(0, 0.3) and +(0, 0.3) .. (0.85, 0.5)-- (0.85, 0) .. controls +(0, -0.3) and +(0, -0.3) .. (0.35, 0)--(0.35, 0.5);
\draw [blue, fill=white] (0,0) rectangle (0.5, 0.5);
\node at (0.25, 0.25) {$x$};
\draw (0.15, 0.5)--(0.15, 0.9) (0.15, 0)--(0.15, -0.4);
\begin{scope}[shift={(0.7, 0)}]
\draw [blue, fill=white] (0,0) rectangle (0.5, 0.5);
\node at (0.25, 0.25) {$y$};
\draw (0.35, 0.5)--(0.35, 0.9)  (0.35, 0)--(0.35, -0.4);
\end{scope}
\draw (0.35, 0.5) .. controls +(0, 0.3) and +(0, 0.3) .. (0.85, 0.5);
\draw (0.35, 0) .. controls +(0, -0.3) and +(0, -0.3) .. (0.85, 0);
\end{tikzpicture}} \;.
\end{align*}
The string Fourier transform $\mathfrak{F}_s$: $\mathscr{P}_{2,\pm}\to \mathscr{P}_{2,\mp}$ is
\begin{align}
\mathfrak{F}_s(x)=\raisebox{-0.9cm}{
\begin{tikzpicture}[scale=1.5]
\path  (-0.3, -0.4) rectangle (0.8, 0.9);
\path [fill=white] (0.35, -0.4)--(0.35, 0.5) .. controls +(0, 0.3) and +(0, 0.3) .. (0.65, 0.5)--(0.65, -0.4);
\path[fill=white] (0.15, 0.9) -- (0.15, 0) .. controls +(0, -0.3) and +(0, -0.3) .. (-0.15, 0)--(-0.15, 0.9);
\draw [blue, fill=white] (0,0) rectangle (0.5, 0.5);
\node at (0.25, 0.25) {$x$};
\draw (0.35, 0)--(0.35, -0.4) (0.15, 0.5)--(0.15, 0.9);
\draw (0.35, 0.5) .. controls +(0, 0.3) and +(0, 0.3) .. (0.65, 0.5)--(0.65, -0.4);
\draw (0.15, 0) .. controls +(0, -0.3) and +(0, -0.3) .. (-0.15, 0)--(-0.15, 0.9);
\end{tikzpicture}}
\end{align}
We refer the readers to \cite{Jon99,BJ00,BJ03,LPR22, Liu16,JLW16,JJLRW20} and references therein for further study on $\mathscr{P}_{2,\pm}$ of subfactor planar algebras.

A fusion bialgebra $(\mathcal{A},\mathcal{B},d,\tau,\mathfrak{F})$ is called subfactorizable if it comes from the quintuple $(\mathscr{P}_{2,+},\mathscr{P}_{2,-},tr_{2,+},tr_{2,-},\mathfrak{F}_s)$ of a subfactor
planar algebra up to gauge parameters $(\lambda_1,\lambda_2)$. 

If a fusion ring $\mathcal{R}$ admits a unitary categorification $\mathscr{C}$, then the quantum double construction of $\mathscr{C}$ produces a subfactor planar algebra which subfactorizes $\mathcal{R}$. 
For readers' convenience, we briefly recall this process in \S \ref{Sec: Quantum Double} in the Appendix.

The convolution is positive on $\sP_{2,\pm}$ for a subfactor planar algebra, called the quantum Schur product theorem \cite[Theorem 4.1]{Liu16}.
However, the convolution may not be positive on $\mathcal{B}$ of a fusion bialgebra. Due to the subfactorization process, the failure of the quantum Schur product theorem turned out to be a surprisingly efficient obstruction of unitary categorification of fusion rings, called the Schur product criterion \cite[Proposition 8.3]{LPW21}.
This criterion could be reformulated for commutative fusion rings as follows \cite[Corollary 8.5]{LPW21}.
\begin{corollary}
Suppose $\mathcal{R}$ is a commutative fusion ring with a basis $\{x_1=1,x_2,\ldots,x_m\}$  and $M_i$ is the fusion matrix of $x_i$. Let $(\lambda_{i,j})_{1\leq i,j\leq m}$ be the character table such that $\lambda_{i,1}=\|M_i\|$. 
If $\mathcal{R}$ admits a unitary categorification, then for any $m$-tuple $\vec{j}=(j_1,\ldots,j_m)$ and $n\geq1$,
\begin{align}\label{eq:n positivity criterion commutative}
\nu_{\vec{j}}:=\sum_{i=1}^m\frac{\prod_{k=1}^n \lambda_{i,j_k}}{\lambda_{i,1}^{n-2}}\geq0.
\end{align}
\end{corollary}
For each $n$, Equation \eqref{eq:n positivity criterion commutative} is equivalent to that the convolution of $n-1$ positive operators is positive. Therefore, the case $n=3$ implies the case $n\geq 3$.
For noncommutative fusion rings, the corresponding statement in terms of irreducible representations of $\mathcal{R}$ is given in \cite[Proposition 8.3]{LPW21}, which is inconvenient to check (by computer) in practice. 

It is worth mentioning that  $\{\nu_{\vec{j}}\}$ are the eigenvalues of the primary matrix for commutative fusion rings. Etingof, Nikshych and Ostrik considered these eigenvalues as invariants of fusion rings and discussed their integrality property in \cite{ENO21}.
It will be interesting to study the integrality property of these eigenvalues of the primary matrix for fusion rings in general.
In particular, if Kaplansky's 6th conjecture holds for spherical categories, then $\{\delta^{2(n-2)} \nu_{\vec{j}}\}$ are algebraic integers for Grothendiec rings of spherical categories.

\section{Primary Criteria}
In this section, we prove a family of criteria for unitary categorification of fusion rings in Theorem \ref{thm:Higher Positivity Criterion}, which are stronger than the Schur product criterion. We also provide concrete examples and computation results to illustrate the efficiency of the criteria.
\subsection{Complete positivity and Fourier Multiple}\label{subsec:Fourier Multiplier and Complete Positivity}
Suppose $\mathcal{A}$ and $\mathcal{B}$ are $C^*$-algebras and  $\mathcal{A}$ is finite dimensional. Let $\tau$ be a faithful trace on $\mathcal{A}$ and $\Omega$ be the vacuum vector in the GNS construction. 
 Let $\mathcal{A}'$ be the commutant algebra of $\mathcal{A}$ on $L^2(\mathcal{A})$. Let $J$ be the modular conjugation, $J(a\Omega)= a^*\Omega$. Then conjugation of $J$ is a map from $\mathcal{A}$ to $\mathcal{A}'$ and $J\mathcal{A}J=\mathcal{A}'$ in Tomita-Takesaki theory
\begin{proposition}\label{prop:2 positive}
Let $\{a_j\}_{j=1}^m$ be an orthonormal basis of $\mathcal{A}$ with respect to $\tau$. Then
\begin{align*}
\sum_{j=1}^m Ja_j J\otimes a_j\geq0
\end{align*}
in $\mathcal{A}'\otimes\mathcal{A}$.
\end{proposition}
\begin{proof}
Define a linear map $T$: $\mathcal{A}'\otimes\mathcal{A}\to\mathbb{C}$,
\begin{align*}
J(x)J\otimes y\mapsto \tau(x^*y).
\end{align*}
Then $T$ is a positive linear functional.  Indeed,
\begin{align*}
T\left(\sum_{j=1}^n J(x_j)J\otimes y_j\right)\left(\sum_{j=1}^n J(x_j)J\otimes y_j\right)^*&=\sum_{i,j=1}^nT( J(x_i x_j^*)J\otimes x_i x_j^*)
\\
&=\sum_{i,j=1}^n T(x_i^*x_i x_j^*x_j)\\
&\geq0.
\end{align*}
It is clear that
\begin{align*}
\left\langle J(x)J\otimes y, \sum_{j=1}^m Ja_j J\otimes a_j\right\rangle&=\sum_{j=1}^m \left\langle J(x)J,J(a_j)J\right\rangle
\left\langle y, a_j\right\rangle\\
&=\sum_{j=1}^m \left\langle a_j,x\right\rangle
\left\langle y, a_j\right\rangle\\
&=\left\langle y,x\right\rangle.
\end{align*}
By Riesz representation theorem and the positivity of $T$, we have $\sum_{j=1}^m Ja_j J\otimes a_j\geq0$.
\end{proof}

\begin{definition}\label{def:Fourier multiple}
For a linear map $\Phi:\mathcal{A} \to \mathcal{B}$, we define its Fourier multiplier $\hat{\Phi}$ in $\mathcal{A}' \otimes \mathcal{B}$ as
\begin{align*}
\hat{\Phi}=\sum_j J(a_j)J \otimes  \Phi(a_j),
\end{align*}
where $\{a_j\}$ is an orthonormal basis of $\mathcal{A}$ w.r.t. $\tau$.
\end{definition}
Note that $\hat{\Phi}$ is independent of the choice of the basis.

\begin{proposition}
We define
\begin{align*}
a*\hat{\Phi}&:= (\Omega^* \otimes I_\mathcal{B})((a\otimes I_\mathcal{B})  \hat{\Phi} ) (\Omega \otimes I_\mathcal{B}),
\end{align*}
where $I_\mathcal{B}$ is the identity of $\mathcal{B}$. 
Then $a*\hat{\Phi}=\Phi(a)$.
\end{proposition}
\begin{proof}
We have
\begin{align*}
a*\hat{\Phi}&= \sum_{j} (\Omega^* \otimes I_\mathcal{B})((a\otimes I_\mathcal{B})  (J(a_j)J \otimes  \Phi(a_j))) (\Omega \otimes I_\mathcal{B}),\\
&= \sum_{j} \Omega^* aJ(a_j)J \Omega \otimes  \Phi(a_j)\\
&= \sum_{j} \Omega^* a a_j^* \Omega \otimes  \Phi(a_j) \\
&= \sum_{j} \tau(a a_j^*) \Phi(a_j) \\
&=\Phi(a) \;.
\end{align*}
\end{proof}

\begin{proposition}\label{Prop: 1}
For a linear map $\Phi:\mathcal{A} \to \mathcal{B}$, we have that $\Phi$ is completely positive if and only if $\hat{\Phi} \geq 0$.
\end{proposition}
\begin{proof}
If $\hat{\Phi} \geq 0$, then $\hat{\Phi}^{1/2}$ is in $\mathcal{A} ' \otimes \mathcal{B} $. We have that
\begin{align*}
\Phi(a)&=(\Omega^* \otimes I_\mathcal{B})((a\otimes I_\mathcal{B})  \hat{\Phi} ) (\Omega \otimes I_\mathcal{B}) \\
&=(\Omega^* \otimes I_\mathcal{B})( \hat{\Phi}^{1/2} (a\otimes I_\mathcal{B})  \hat{\Phi}^{1/2} ) (\Omega \otimes I_\mathcal{B}).
\end{align*}
So $\Phi$ is completely positive. On the other hand,
\begin{align*}
\hat{\Phi}=I\otimes \Phi\left(\sum_{j=1}^m Ja_j J\otimes a_j\right).
\end{align*}
By Proposition \ref{prop:2 positive}, $\hat{\Phi}\geq0$.
\end{proof}
\subsection{Comultiplication and Primary Matrix}
Suppose $\mathcal{R}$ is a fusion ring with basis $\{x_1=1, x_2,\ldots,x_m\}$ and $M_i$ is the fusion matrix of $x_i$. Let $(\mathcal{A},\mathcal{B},d,\tau,\mathfrak{F})$ be the fusion bialgebra arising from $\mathcal{R}$.
\begin{definition}
 Let $(\mathcal{A},\mathcal{B},d,\tau,\mathfrak{F})$ be a fusion bialgebra. We define the \textbf{comultiplication} 
$\Delta$:  $\mathcal{B}\rightarrow\mathcal{B}\otimes\mathcal{B}$ as a linear map such that
\begin{align*}
\Delta(x_j)=\frac{1}{d_j} x_j\otimes x_j,
\end{align*}
where $d_j$ is the quantum dimension of $x_j$.
Moreover, we define the \textbf{higher comultiplication} $\Delta^{(n)}$: $\mathcal{B}\rightarrow\mathcal{B}^{\otimes (n+1)}$, $n\geq2$, as a linear map such that
\begin{align*}
\Delta^{(n)}(x_j)=\frac{1}{d_j^{n}}x_j^{\otimes (n+1)}.
\end{align*}
\end{definition}

 The convolution on $\mathcal{B}$ in \eqref{def:convolution} induces a linear map $\mathcal{B}\otimes\mathcal{B}\to\mathcal{B}$:
$\sum a_i\otimes b_i\mapsto\sum a_i\ast b_i$. 
It is clear that 
\begin{align}
\left\langle x_i\ast x_j,x_k\right\rangle=\left\langle x_i\otimes x_j,\Delta(x_k)\right\rangle.
\end{align}
So the comultiplication is the dual operator of this linear map.
\begin{proposition}\label{prop:comultiplication positive imply convolution positive fusion bialgebra}
If the comultiplication $\Delta$:  $\mathcal{B}\rightarrow\mathcal{B}\otimes\mathcal{B}$ is positive then
the convolution is also positive, i.e., $x\ast y\geq0$, for $x,y\geq0$, $x,y\in\mathcal{B}$.
\end{proposition}
\begin{proof}
For any $x,y,z\geq0$, $x,y,z\in\mathcal{B}$, we have 
\begin{align}
\left\langle x\ast y,z\right\rangle=\left\langle x\otimes y,\Delta(z)\right\rangle\geq0.
\end{align}
So $x\ast y\geq0$.
\end{proof}
\begin{definition}
 Let $(\mathcal{A},\mathcal{B},d,\tau,\mathfrak{F})$ be a fusion bialgebra.
For any $n\geq1$, we define 
\begin{align}\label{eq:positive criterion matrix}
T_n(\mathcal{B})=\sum_{i=1}^m \|M_i\|^2\bigg(\frac{M_i}{\|M_i\|}\bigg)^{\otimes n},
\end{align}
where $\|M_i\|$ is the operator norm of $M_i$ acting on $\mathbb{C}^m$.
We call it the \textbf{primary $n$-matrix} of $\mathcal{B}$. 
For simplicity, we use $T_n$ instead $T_n(\mathcal{B})$ if there is no confusion.
\end{definition}

\begin{proposition}\label{prop:1 positivity}
 Let $(\mathcal{A},\mathcal{B},d,\tau,\mathfrak{F})$ be a fusion bialgebra. Then the primary 1-matrix is positive.
\end{proposition}
\begin{proof}
Let $e=\sum_{j=1}^m d_j x_j$. Then $e=e^*$ and $e^2=\text{\rm FPdim}_\mathbb{C}(\mathcal{R}) e$. So $e$
is a multiple of a projection. Thus $e\geq0$. The representation of $e$ on the basis $\{x_j\}$ is
\begin{align*}
\sum_{j=1}^m \|M_j\|M_j.
\end{align*}
So it is positive.
\end{proof}

\begin{proposition}\label{prop:2 positivity}
 Let $(\mathcal{A},\mathcal{B},d,\tau,\mathfrak{F})$ be a fusion bialgebra. Then the primary 2-matrix is positive.
\end{proposition}
\begin{proof}
By Proposition \ref{prop:2 positive}, we have
\begin{align*}
\sum_{j=1}^m Jx_j J\otimes x_j\geq0.
\end{align*}
The representation of this operator on the basis $\{J(x_i)J\otimes x_j\}$ is
\begin{align*}
\sum_{j=1}^m M_j\otimes M_j.
\end{align*}
So it is positive.
\end{proof}

\begin{proposition} Let $(\mathcal{A},\mathcal{B},d,\tau,\mathfrak{F})$ be a fusion bialgebra. Then the comultiplication $\Delta: \mathcal{R} \to \mathcal{R} \otimes \mathcal{R}$ is completely positive if and only if  the primary $3$-matrix is positive. If 
the comultiplication $\Delta$ is completely positive, then the primary $n$-matrix is positive for all $n$.
\end{proposition}
\begin{proof}
For the comultiplication $\Delta: \mathcal{R} \to \mathcal{R} \otimes \mathcal{R}$, its Fourier multiplier is an operator in $\mathcal{R}' \otimes  \mathcal{R} \otimes \mathcal{R}$. We have
\begin{align*}
\hat{\Delta}&=\sum_{j=1}^m  J(x_j)J \otimes \Delta(x_j) =\sum_{j=1}^m d_j^{-1} J(x_j)J \otimes x_j \otimes x_j.\\
\end{align*}
The representation of $\hat{\Delta}$ on the basis $\{ J(x_i)J \Omega \otimes x_j \Omega  \otimes x_k \Omega\}$ is 
\begin{align*}
\sum_{j=1}^m \|M_j\|^{-1} M_j\otimes M_j \otimes M_j.
\end{align*}
By Proposition \ref{Prop: 1}, if the primary 3-matrix is positive, then $\Delta$ is completely positive. 
Then $\Delta^{(n)}$ is a positive map and $\Delta^{(n)}(e)$ is a positive operator. Therefore, the primary $n$-matrix as a representation of $\Delta^{n-1}(e)$ on the basis is positive.
\end{proof}
\begin{proposition} Let $(\mathcal{A},\mathcal{B},d,\tau,\mathfrak{F})$ be a fusion bialgebra. If  the primary $n$-matrix is positive, then the primary $(n-1)$-matrix is also positive.
\end{proposition}
\begin{proof}
Recall that the representation of $e$ on the basis $\{x_j\}$ is $\sum_{j=1}^m \|M_j\|M_j$, denoted by $M$. 
Then $M\geq0$ and $M_i M=d_i M$. Let ${\rm Tr}_1$ be the partial trace from $M_m(\mathbb{C})^{\otimes n}$ onto 
$M_m(\mathbb{C})^{\otimes (n-1)}$. Then 
\begin{align*}
{\rm Tr}_1[(M\otimes I) T_n (M\otimes I)]={\rm Tr}(M^2)T_{n-1}\geq0.
\end{align*}
So the primary $(n-1)$-matrix is positive.
\end{proof}
In summary, we have
\begin{theorem}\label{thm:complete positive primary matrix}
Let $(\mathcal{A},\mathcal{B},d,\tau,\mathfrak{F})$ be a fusion bialgebra. The following statements are equivalent:
\begin{enumerate}
\item $\Delta$ is completely positive;
\item the primary 3-matrix is positive;
\item the primary $n$-matrix is positive for some $n\geq 3$;
\item the primary $n$-matrix is positive for any $n \in \mathbb{N}$.
\end{enumerate}
\end{theorem}
\subsection{Main Results}
Suppose $\mathscr{P}$ is a subfactor planar algebra with index $\delta^2$. We define the comultiplication $\Delta$: $\mathscr{P}_{2,-}\to\mathscr{P}_{2,-}\otimes \mathscr{P}_{2,-}$ as a linear map such that
\begin{align}\label{eq:definition of comultiplication on subfactor}
\left\langle \Delta(z),x\otimes y\right\rangle=\left\langle z,x\ast y\right\rangle,\quad \forall x,y,z\in\mathscr{P}_{2,-}.
\end{align}
Switching the input discs and the output disc of the convolution tangle, we obtain the following surface tangle representing the comultiplication, see  \cite{Liu19} for the theory of surface tangles and surface algebras:
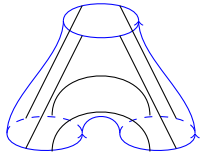
\begin{align*}
\begin{tikzpicture}
\begin{scope}
\draw [blue, ->] (0,0).. controls +(0, -0.3) and +(0, -0.3) .. (1, 0);
\draw [blue, dashed] (0,0).. controls +(0, 0.3) and +(0, 0.3) .. (1, 0);
\end{scope}
\begin{scope}[shift={(1.5, 0)}]
\draw [blue, ->] (0,0).. controls +(0, -0.3) and +(0, -0.3) .. (1, 0);
\draw [blue, dashed] (0,0).. controls +(0, 0.3) and +(0, 0.3) .. (1, 0);
\end{scope}
\begin{scope}[shift={(0.75, 1.5)}]
\draw [blue, ->] (0,0).. controls +(0, -0.3) and +(0, -0.3) .. (1, 0);
\draw [blue] (0,0).. controls +(0, 0.3) and +(0, 0.3) .. (1, 0);
\end{scope}
\draw [blue] (0,0) .. controls +(0, 0.5) and +(0, -0.5) .. (0.75, 1.5);
\draw [blue] (2.5,0) .. controls +(0, 0.5) and +(0, -0.5) .. (1.75, 1.5);
\draw [blue] (1,0) .. controls +(0, 0.3) and +(0, 0.3) .. (1.5, 0);
\draw (0.25,0.2)--(1, 1.7) (0.25,-0.2)--(1, 1.3) (2.25,-0.2)--(1.5, 1.3) (2.25,0.2)--(1.5, 1.7);
\draw (0.6, -0.24)  .. controls +(0, 0.7) and +(0, 0.7) .. (1.9, -0.24);
\draw (0.6, 0.24)  .. controls +(0, 0.7) and +(0, 0.7) .. (1.9, 0.24);
\end{tikzpicture}
\end{align*}
\begin{theorem}\label{thm:completely positive comultiplication}
Suppose $\mathscr{P}$ is a subfactor planar algebra. Then the comultiplication $\Delta$: $\mathscr{P}_{2,-}\to\mathscr{P}_{2,-}\otimes\mathscr{P}_{2,-}$ is completely positive.
\end{theorem}
\begin{proof}
From Equation \eqref{eq:definition of comultiplication on subfactor}, we have
\begin{align*}
tr_{2,-}\left(\raisebox{-1.5cm}{
\begin{tikzpicture}[scale=1.5]
\draw [blue, fill=white] (0,0) rectangle (0.5, 0.5);
\node at (0.25, 0.25) {$x$};
\draw (0.15, 0.5)--(0.15, 0.9) (0.15, 0)--(0.15, -0.5);
\draw (0.35, 0.5)--(0.35, 0.9) (0.35, 0)--(0.35, -0.5);
\begin{scope}[shift={(0.7, 0)}]
\draw [blue, fill=white] (0,0) rectangle (0.5, 0.5);
\node at (0.25, 0.25) {$y$};
\draw (0.15, 0.5)--(0.15, 0.9) (0.15, 0)--(0.15, -0.5);
\draw (0.35, 0.5)--(0.35, 0.9)  (0.35, 0)--(0.35, -0.5);
\end{scope}
\begin{scope}[shift={(0, -0.9)}]
\draw [blue, fill=white] (0,0) rectangle (1.2, 0.5);
\node at (0.6, 0.25) {$\Delta(z)$};
\draw (0.15, 0.5)--(0.15, 0.9)  (0.15, 0)--(0.15, -0.4);
\draw (0.35, 0.5)--(0.35, 0.9)  (0.35, 0)--(0.35, -0.4);
\draw (0.85, 0.5)--(0.85, 0.9)  (0.85, 0)--(0.85, -0.4);
\draw (1.05, 0.5)--(1.05, 0.9)  (1.05, 0)--(1.05, -0.4);
\end{scope}
\end{tikzpicture}}\right)
= tr_{2,-}\left(\raisebox{-1.5cm}{
\begin{tikzpicture}[scale=1.5]
\path [fill=white] (0.35, 0.5) .. controls +(0, 0.3) and +(0, 0.3) .. (0.85, 0.5)-- (0.85, 0) .. controls +(0, -0.3) and +(0, -0.3) .. (0.35, 0)--(0.35, 0.5);
\draw [blue, fill=white] (0,0) rectangle (0.5, 0.5);
\node at (0.25, 0.25) {$x$};
\draw (0.15, 0.5)--(0.15, 0.9) (0.15, 0)--(0.15, -0.5);
\begin{scope}[shift={(0.7, 0)}]
\draw [blue, fill=white] (0,0) rectangle (0.5, 0.5);
\node at (0.25, 0.25) {$y$};
\draw (0.35, 0.5)--(0.35, 0.9)  (0.35, 0)--(0.35, -0.5);
\end{scope}
\begin{scope}[shift={(0, -0.9)}]
\draw [blue, fill=white] (0,0) rectangle (1.2, 0.5);
\node at (0.6, 0.25) {$z$};
\draw (0.15, 0.5)--(0.15, 0.9)  (0.15, 0)--(0.15, -0.4);
\draw (1.05, 0.5)--(1.05, 0.9)  (1.05, 0)--(1.05, -0.4);
\end{scope}
\draw (0.35, 0.5) .. controls +(0, 0.3) and +(0, 0.3) .. (0.85, 0.5);
\draw (0.35, 0) .. controls +(0, -0.3) and +(0, -0.3) .. (0.85, 0);
\end{tikzpicture}}\right)=
tr_{2,-}\left(\raisebox{-1.5cm}{
\begin{tikzpicture}[scale=1.5]
\path [fill=white] (0.35, 0.5) .. controls +(0, 0.3) and +(0, 0.3) .. (0.85, 0.5)-- (0.85, 0) .. controls +(0, -0.3) and +(0, -0.3) .. (0.35, 0)--(0.35, 0.5);
\draw [blue, fill=white] (0,0) rectangle (0.5, 0.5);
\node at (0.25, 0.25) {$x$};
\draw (0.15, 0.5)--(0.15, 0.9) (0.15, 0)--(0.15, -0.5);
\draw (0.35, 0.5)--(0.35, 0.9) ;
\begin{scope}[shift={(0.7, 0)}]
\draw [blue, fill=white] (0,0) rectangle (0.5, 0.5);
\node at (0.25, 0.25) {$y$};
\draw (0.35, 0.5)--(0.35, 0.9)  (0.35, 0)--(0.35, -0.5);
\draw (0.15, 0.5)--(0.15, 0.9);
\end{scope}
\begin{scope}[shift={(0, -0.9)}]
\draw [blue, fill=white] (0,0) rectangle (1.2, 0.5);
\node at (0.6, 0.25) {$z$};
\draw (0.15, 0.5)--(0.15, 0.9)  (0.15, 0)--(0.15, -0.4);
\draw (1.05, 0.5)--(1.05, 0.9)  (1.05, 0)--(1.05, -0.4);
\end{scope}
\draw (0.35, 0) .. controls +(0, -0.3) and +(0, -0.3) .. (0.85, 0);
\begin{scope}[shift={(0, -1.8)}]
\draw (0.35, 0.5) .. controls +(0, 0.3) and +(0, 0.3) .. (0.85, 0.5);
\end{scope}
\end{tikzpicture}}\right)
\end{align*}
Let 
\begin{align*}
\Psi: \sP_{4,\pm} \to \sP_{2,\pm}^{\otimes 2}
\end{align*}
be the trace-preserving conditional expectation. 
Then one could obtain that $\Delta$ has graphical
representation as follows:
\begin{align}\label{eq:comultiplication}
\Delta(z)=\Psi\left(
\raisebox{-0.8cm}{
\begin{tikzpicture}
\draw (0.9, -0.4)--(0.1, 1.4) (0.1, -0.4)--(0.9, 1.4) ; 
\draw [blue, fill=white] (0, 0) rectangle (1, 1);
\node at (0.5, 0.5) {$z$};
\draw (0.3, -.4).. controls + (0, 0.3) and +(0, 0.3) .. (0.7, -0.4);
\draw (0.3, 1.4).. controls + (0, -0.3) and +(0, -0.3) .. (0.7, 1.4);
\end{tikzpicture}}\right).
\end{align}
This indicates that $\Delta$ is a composition (up to a scalar) of a $*$-isomorphism and a conditional expectation. So $\Delta$ is completely positive.
\end{proof}
\begin{remark}\label{rem:comultiplication stronger than convolution}
The  positivity of the comultiplication $\Delta$: $\mathscr{P}_{2,-}\to\mathscr{P}_{2,-}\otimes\mathscr{P}_{2,-}$ indicates the positivity of the convolution: $x\ast y\geq0$, for $x,y\geq0$, $x,y\in\mathscr{P}_{2,-}$, the Schur product theorem \cite[Theorem 4.1]{Liu16}.
\end{remark}
\begin{proposition}\label{prop:subfactorized imply complete positivity}
Let $(\mathcal{A},\mathcal{B},d,\tau,\mathfrak{F})$ be a fusion bialgebra. If it is subfactorizable, then the comultiplication $\Delta$:
$\mathcal{B}\to \mathcal{B}\otimes\mathcal{B}$ is completely positive and the primary $n$-matrix in \eqref{eq:positive criterion matrix} is positive for all $n$. 
\end{proposition}
\begin{proof}
Suppose the fusion bialgebra arises from a subfactor planar algebra $\mathscr{P}$.
Then the comultiplication $\Delta$: $\mathcal{B}\to \mathcal{B}\otimes\mathcal{B}$ is consistent with the comultiplication $\Delta$: $\mathscr{P}_{2,-}\to \mathscr{P}_{2,-}\otimes\mathscr{P}_{2,-}$. So by Theorem \ref{thm:completely positive comultiplication}, it is completely positive. By Theorem \ref{thm:complete positive primary matrix}, the primary $n$-matrix is positive for all $n$.
\end{proof}
A fusion ring $\mathcal{R}$ admits a unitary categorification means it is the Grothendieck ring of a unitary fusion category. 
For the definition of fusion category, we refer the readers to  \cite{ENO05} and \cite[Section 1.12]{EGNO15}. 
The Grothendieck ring of a fusion category is a fusion ring.
A unitary fusion category is a fusion category with a unitary structure (See e.g. \cite[Remark 9.4.7]{EGNO15}). If $\mathcal{R}$ is the Grothendieck ring of a unitary fusion category, then the canonical fusion bialgebra associated to the fusion ring is subfactorizable by quantum double construction \cite[Proposition 7.4]{LPW21}. 

Let $(\mathscr{P}_{2,+},\mathscr{P}_{2,-},{\rm Tr}_{2,+},tr_{2,-},\mathfrak{F})$ be the canonical fusion bialgebra associated to the canonical Frobenius algebra $\gamma$ of $\mathscr{C}\otimes \overline{\mathscr{C}}$
in the quantum double construction, see Appendix \ref{Sec: Principal Graph}. Here  ${\rm Tr}_{2,+}$ is the unnormalized Markov trace on $\mathscr{P}_{2,+}$ and $\mathfrak{F}=\delta\mathfrak{F}_s$.
From Equations \eqref{eq:definition of comultiplication on subfactor} and \eqref{eq: 1}, we have that the comultiplication $\Delta$: $\mathscr{P}_{2,-}\to\mathscr{P}_{2,-}\otimes \mathscr{P}_{2,-}$ satisfying
\begin{align*}
\left\langle \Delta(\mathfrak{F}(\beta_j)) ,\mathfrak{F}(\beta_k)\otimes \mathfrak{F}(\beta_i)\right\rangle
&=\left\langle \mathfrak{F}(\beta_j) ,\mathfrak{F}(\beta_k)\ast \mathfrak{F}(\beta_i)\right\rangle\\
&=\dfrac{\delta_{k,i}}{d_i}\left\langle \mathfrak{F}(\beta_j) ,\mathfrak{F}(\beta_k)\right\rangle\\
&=\dfrac{\delta_{k,i}\delta_{jk}}{d_i}.
\end{align*}
We see that
\begin{align}\label{eq: 3}
\Delta(\fF(\beta_j))
=\frac{1}{d_j} \fF(\beta_j)\otimes \fF(\beta_j).
\end{align}
Let $(\mathcal{A},\mathcal{B},d,\tau,\mathfrak{F})$ be the canonical fusion bialgebra associated to the Grothendieck ring $\mathcal{R}$ of a unitary fusion category $\mathscr{C}$. Then it is isomorphic to $(\mathscr{P}_{2,+},\mathscr{P}_{2,-},{\rm Tr}_{2,+},tr_{2,-},\mathfrak{F})$ by mapping $x_j$ to $\beta_j$ \cite[Proposition 7.4]{LPW21}.
So the comultiplication $\Delta$: $\mathscr{P}_{2,-}\to\mathscr{P}_{2,-}\otimes \mathscr{P}_{2,-}$ and the comultiplication 
$\Delta$: $\mathcal{B}\to\mathcal{B}\otimes\mathcal{B}$ are consistent. 
\begin{theorem}[Primary criteria]\label{thm:Higher Positivity Criterion}
Suppose $\mathcal{R}$ is a fusion ring with basis $\{x_1=1, x_2,\ldots,x_m\}$  and $M_j$ is the fusion matrix of $x_j$. 
If $\mathcal{R}$ admits a unitary categorification, then for any $n \geq 1$, we have
\begin{align}\label{eq:n positivity criterion}
\sum_{j=1}^m \|M_j\|^2\bigg(\frac{M_j}{\|M_j\|}\bigg)^{\otimes n}\geq 0.
\end{align}
\end{theorem}
\begin{proof}
Let $(\mathcal{A},\mathcal{B},d,\tau,\mathfrak{F})$ be the fusion bialgebra associated to $\mathcal{R}$. Then it is subfactorizable. So the comultiplication $\Delta$: $\mathcal{B}\to\mathcal{B}\otimes \mathcal{B}$ is completely positive by Proposition \ref{prop:subfactorized imply complete positivity}. Further by Theorem \ref{thm:complete positive primary matrix}, the primary $n$-matrix is positive for all $n$.
\end{proof}

\begin{definition}
For any $n\geq1$,  we call Inequality (\ref{eq:n positivity criterion})  as \textbf{primary $n$-criterion} of unitary categorification of fusion rings.
\end{definition}

\begin{remark}
Though primary $n$-criterion are equivalent when $n\geq3$, we  will show that they are not locally equivalent
in \S \ref{sec:reduced twist criteria}. 
\end{remark}

\begin{proposition}
Suppose $\mathcal{R}$ is a fusion ring. If $\mathcal{R}$ passes
the primary $3$-criterion  then $\mathcal{R}$ passes the Schur product criterion.
\end{proposition}
\begin{proof}
Let $(\mathcal{A},\mathcal{B},d,\tau,\mathfrak{F})$ be the fusion bialgebra arising from the fusion ring $\mathcal{R}$. 
If $\mathcal{R}$ passes the primary $3$-criterion, then the primary 3-matrix is positive.
Theorem \ref{thm:complete positive primary matrix} indicates that the comultiplication $\Delta:\mathcal{B}\to\mathcal{B}\otimes\mathcal{B}$ is completely positive.
By Proposition \ref{prop:comultiplication positive imply convolution positive fusion bialgebra}, we have $x*y \geq 0$ for any $x, y\geq 0$, $x,y\in\mathcal{B}$. So $\mathcal{R}$ passes the Schur product criterion.
\end{proof}

In \cite{VS22}, many fusion rings of low rank were produced by computers. One of the three simple non-commutative fusion rings 
 in the dataset of \cite{VS22} can be excluded from unitary categorification by primary $3$-criterion. 
\begin{proposition} \label{prop:nc6}
Let $\mathcal{R}_6$ be the following non-commutative simple fusion ring  of rank $6$, type $[[1,1],[7+2\sqrt{13},3],[ 11+3\sqrt{13},2]]$ with fusion matrices:
\begin{align*}
\begin{smallmatrix}
1&0&0&0&0&0\\
0&1&0&0&0&0\\
0&0&1&0&0&0\\
0&0&0&1&0&0\\
0&0&0&0&1&0\\
0&0&0&0&0&1
\end{smallmatrix},
\begin{smallmatrix}
0&1&0&0&0&0\\
1&4&2&2&2&2\\
0&2&2&1&2&4\\
0&2&1&2&4&2\\
0&2&2&4&5&4\\
0&2&4&2&4&5
\end{smallmatrix},
\begin{smallmatrix}
0&0&1&0&0&0\\
0&2&2&1&4&2\\
0&1&3&1&3&3\\
1&2&3&3&1&3\\
0&2&3&3&5&4\\
0&4&1&3&4&5
\end{smallmatrix},
\begin{smallmatrix}
0&0&0&1&0&0\\
0&2&1&2&2&4\\
1&2&3&3&3&1\\
0&1&1&3&3&3\\
0&4&3&1&5&4\\
0&2&3&3&4&5
\end{smallmatrix},
\begin{smallmatrix}
0&0&0&0&1&0\\
0&2&4&2&5&4\\
0&4&1&3&5&4\\
0&2&3&3&5&4\\
1&5&5&5&5&7\\
0&4&4&4&7&7
\end{smallmatrix},
\begin{smallmatrix}
0&0&0&0&0&1\\
0&2&2&4&4&5\\
0&2&3&3&4&5\\
0&4&3&1&4&5\\
0&4&4&4&7&7\\
1&5&5&5&7&5
\end{smallmatrix}
\end{align*}
Then $\mathcal{R}_6$ admits no unitary categorification.
\end{proposition}
\begin{proof}
The primary $3$-matrix $T_3$ of $\mathcal{R}_6$ has a negative eigenvalue $\simeq -1.176375$ (See Appendix \ref{App:SNC6} for SageMath code). 
So $\mathcal{R}_6$ admits no unitary categorification by Theorem \ref{thm:Higher Positivity Criterion}.
\end{proof}
Moreover, we complete the full computation on their dataset in \cite{VS22}:
\begin{itemize}
\item Among the 28451 fusion rings, exactly 19738 does not pass primary 3-criterion (about 69.37$\%$);
\item Among the 14558 non-simple ones, exactly 8009 does not pass primary 3-criterion (about 55.01$\%$);
\item Among the 13893 simple ones, exactly 11729 ones does not pass primary 3-criterion (about 84.4$\%$).
\end{itemize}
It turns out that primary $3$-criterion is more efficient for simple fusion rings.

\begin{remark}
The primary criteria in Inequality \eqref{eq:n positivity criterion}  are very easy to check, and friendly to computers for all fusion rings including non-commutative cases. It is hard to check non-commutative fusion rings by the Schur product criterion (See Proposition 8.3 in \cite{LPW21}) since we have to know all the irreducible unital $*$-representations of the non-commutative fusion rings. So checking noncommutative fusion rings is an advantage of the primary criteria.
\end{remark}
\begin{remark}
We need to compute the full character table of fusion matrices to apply the Schur product criterion. But the full  fusion rules
are not always known, so the Schur product criterion may not be valid all the time. In \S \ref{sec:localized criteria} and \S \ref{sec:reduced twist criteria}, we introduce
localized versions of primary
criteria to deal with this problem, which could check fusion rings with sparse data.
\end{remark}

\section{Localized Criteria}\label{sec:localized criteria}
In this section, we localize the primary criteria via the basic fact that submatrices inherit the positivity of matrices. In this way, we only need to check the positivity of a small submatrix instead of  a full large matrix. This would improve the efficiency of calculations. Moreover, this localized method makes it applicable to fusion rings with incomplete data, which is superior to the Schur product criterion.
\subsection{Results and Definitions}
\begin{theorem}[Localized criteria]\label{thm:localization criterion}
Suppose $\mathcal{R}$ is a fusion ring with basis $\{x_1=1,x_2,\ldots,x_m\}$  and $M_i$ is the fusion matrix of $x_i$. Let
$S\subseteq \{1,2,\ldots,m\}$ be a subset.
 If $\mathcal{R}$ admits a unitary categorification, then
\begin{align}\label{eq:n positivity locolization criterion}
\sum_{i=1}^m \|M_i\|^2\bigg(\frac{M_i^S}{\|M_i\|}\bigg)^{\otimes n}\geq0,\quad \forall n\geq1,
\end{align}
where $M_i^S$ is the sub-matrix of $M_i$ with rows and columns in $S$.
\end{theorem}
\begin{proof}
Let $P_S$: $\mathbb{C}^m\to\mathbb{C}^{|S|}$ be the projection.
By Inequality \eqref{eq:n positivity criterion}, we have
\begin{align*}
\sum_{i=1}^m \|M_i\|^2\bigg(\frac{M_i^S}{\|M_i\|}\bigg)^{\otimes n}=P_S^{\otimes n}\sum_{i=1}^m \|M_i\|^2\bigg(\frac{M_i}{\|M_i\|}\bigg)^{\otimes n} P_S^{*\otimes n}
\geq0.
\end{align*}
This completes the proof of the theorem.
\end{proof}
\begin{definition}
For any $n \geq 1$ and $S\subseteq\{1,2,\ldots,m\}$, we call Inequality \eqref{eq:n positivity locolization criterion} \textbf{$S$-localized $n$-criterion} of unitary categorification of fusion rings.
\end{definition}
\begin{definition}
For any $n \geq 1$ and $S\subseteq\{1,2,\ldots,m\}$, we call
\begin{align}\label{eq:positivity localization criterion matrix}
T_n^S(\mathcal{R})=\sum_{i=1}^m \|M_i\|^2\bigg(\frac{M_i^S}{\|M_i\|}\bigg)^{\otimes n}
\end{align}
primary \textbf{$(S,n)$-matrix} of $\mathcal{R}$. 
For simplicity, we use $T_n^S$ instead of $T_n^S(\mathcal{R})$ if there is no confusion.
\end{definition}
\subsection{$\mathcal{K}_7$ Revisited and Perturbations} \label{sub:S7}
We denote by $\mathcal{K}_7$ the simple integral fusion ring of rank $7$, FPdim $210$, type [[1, 1], [5, 3], [6, 1], [7, 2]], with
fusion matrices (See page 54 in \cite{LPW21}):
\begin{align*}
\begin{smallmatrix}
1&0&0&0&0&0&0\\
0&1&0&0&0&0&0\\
0&0&1&0&0&0&0\\
0&0&0&1&0&0&0\\
0&0&0&0&1&0&0\\
0&0&0&0&0&1&0\\
0&0&0&0&0&0&1
\end{smallmatrix},
\begin{smallmatrix}
0&1&0&0&0&0&0\\
1&1&0&1&0&1&1\\
0&0&1&0&1&1&1\\
0&1&0&0&1&1&1\\
0&0&1&1&1&1&1\\
0&1&1&1&1&1&1\\
0&1&1&1&1&1&1
\end{smallmatrix},
\begin{smallmatrix}
0&0&1&0&0&0&0\\
0&0&1&0&1&1&1\\
1&1&1&0&0&1&1\\
0&0&0&1&1&1&1\\
0&1&0&1&1&1&1\\
0&1&1&1&1&1&1\\
0&1&1&1&1&1&1
\end{smallmatrix},
\begin{smallmatrix}
0&0&0&1&0&0&0\\
0&1&0&0&1&1&1\\
0&0&0&1&1&1&1\\
1&0&1&1&0&1&1\\
0&1&1&0&1&1&1\\
0&1&1&1&1&1&1\\
0&1&1&1&1&1&1
\end{smallmatrix},
\begin{smallmatrix}
0&0&0&0&1&0&0\\
0&0&1&1&1&1&1\\
0&1&0&1&1&1&1\\
0&1&1&0&1&1&1\\
1&1&1&1&1&1&1\\
0&1&1&1&1&2&1\\
0&1&1&1&1&1&2
\end{smallmatrix},
\begin{smallmatrix}
0&0&0&0&0&1&0\\
0&1&1&1&1&1&1\\
0&1&1&1&1&1&1\\
0&1&1&1&1&1&1\\
0&1&1&1&1&2&1\\
1&1&1&1&2&0&3\\
0&1&1&1&1&3&1
\end{smallmatrix},
\begin{smallmatrix}
0&0&0&0&0&0&1\\
0&1&1&1&1&1&1\\
0&1&1&1&1&1&1\\
0&1&1&1&1&1&1\\
0&1&1&1&1&1&2\\
0&1&1&1&1&3&1\\
1&1&1&1&2&1&2
\end{smallmatrix}
\end{align*}
With the character table of these fusion matrices, $\mathcal{K}_7$ is  excluded from unitary categorification
by the Schur product criterion  in \cite{LPW21}. 

In this subsection, we exclude $\mathcal{K}_7$ again
 by the localized criterion with a local set with only two points. In this way, the computations are much simpler.
 Let $S=\{x_6,x_7\}$.
Then
\begin{align*}
M_1^S=\begin{pmatrix}
1&0\\0&1
\end{pmatrix},M_2^S=M_3^S=M_4^S=\begin{pmatrix}
1&1\\1&1
\end{pmatrix},M_5^S=\begin{pmatrix}
2&1\\1&2
\end{pmatrix},M_6^S=\begin{pmatrix}
0&3\\3&1
\end{pmatrix},M_7^S=\begin{pmatrix}
3&1\\1&2
\end{pmatrix}.
\end{align*}
The primary $(S,3)$-matrix $T_3^S$ of $\mathcal{K}_7$ is
\begin{align*}
\frac{1}{210}\begin{pmatrix}
1426&  536  & 536 &  286&   536 &  286 &  286&  1001\\
  536&  1156  & 286 &  446 &  286 &  446 & 1001 &  526\\
  536 &  286&  1156 &  446 &  286&  1001  & 446 &  526\\
  286&   446  & 446 &  976&  1001 &  526 &  526 &  476\\
  536&   286&   286&  1001&  1156 &  446  & 446 &  526\\
 286&   446&  1001 &  526 &  446 &  976  & 526 &  476\\
  286&  1001&   446 &  526 &  446 &  526 &  976 &  476\\
 1001&   526 &  526&   476 &  526&   476  & 476 &  886
\end{pmatrix}.
\end{align*}
It is not positive since it has a negative eigenvalue $\simeq
-0.62949$ (see Appendix \ref{App:S7}). So $\mathcal{K}_7$ does not admit a unitary categorification by Theorem \ref{thm:localization criterion}.
\subsubsection{Perturbation of Frobenius-Perron dimensions}\label{sec:fpdim}
Localized criteria allow us to do some perturbations of the Frobenius-Perron dimensions. We  obtain a bound of FPdim, which provides
 a sufficient condition for a fusion ring's exclusion from unitary categorification. This helps us to exclude another nine non-integral fusion rings  from unitary categorification.

Suppose that $\mathcal{R}$ is a fusion ring with basis $\{I=x_1,\cdots,x_7\}$ and $M_i$ are fusion matrices  such that
\begin{align*}
M_1^S=\begin{pmatrix}
1&0\\0&1
\end{pmatrix},M_2^S=M_3^S=M_4^S=\begin{pmatrix}
1&1\\1&1
\end{pmatrix},M_5^S=\begin{pmatrix}
2&1\\1&2
\end{pmatrix},M_6^S=\begin{pmatrix}
0&3\\3&1
\end{pmatrix},M_7^S=\begin{pmatrix}
3&1\\1&2
\end{pmatrix},
\end{align*}
where $S=\{x_6,x_7\}$.
We assume that $d_j$ is the Frobenius-Perron dimension and they are variables here.
Let 
\begin{align*}
T_3^S= \sum_{i=1}^7 \frac{1}{d_i}(M^S_i)^{\otimes 3}
\end{align*}
where $d_i \ge 1$ and $d_1 = 1$. 
The exclusion of fusion ring $\mathcal{K}_7$ can be seen as the application of the following proposition to $(d_1, \dots, d_7) = (1,5,5,5,6,7,7)$.

\begin{proposition} \label{prop:1}
The matrix $T_3^S$ has a negative eigenvalue if
 $$d_6 < \max_{i=1,2} Q_{i,+}(d_2,d_3,d_4,d_5,d_7),$$ 
 where $\displaystyle Q_{i,\pm} = \frac{-B_i \pm \sqrt{\Delta_i}}{2A_i}$, $\Delta_i = B_i^2-4A_iC_i$, and
\begin{align*}
A_1 =& d_5^2d_7^2 + 25d_5^2d_7 + 12d_5d_7^2 + 125d_5^2 + 120d_5d_7 + 27d_7^2,\\
B_1 =& - 9(d_5d_7 - 15d_5 - 12d_7)d_5d_7,\\
C_1 =& - 729d_5^2d_7^2,\\
A_2 =& d_2d_3d_4d_5^3d_7^3 + 67d_2d_3d_4d_5^3d_7^2 + 32d_2d_3d_4d_5^2d_7^3 + 7d_2d_3d_5^3d_7^3 + 7d_2d_4d_5^3d_7^3 + 7d_3d_4d_5^3d_7^3 \\
+& 1025d_2d_3d_4d_5^3d_7
+962d_2d_3d_4d_5^2d_7^2
+172d_2d_3d_5^3d_7^2 + 172d_2d_4d_5^3d_7^2 + 172d_3d_4d_5^3d_7^2\\
 +& 195d_2d_3d_4d_5d_7^3
+ 54d_2d_3d_5^2d_7^3 + 54d_2d_4d_5^2d_7^3 + 54d_3d_4d_5^2d_7^3 + 3375d_2d_3d_4d_5^3
+ 4600d_2d_3d_4d_5^2d_7 \\
+ &665d_2d_3d_5^3d_7 + 665d_2d_4d_5^3d_7
+ 665d_3d_4d_5^3d_7 + 1833d_2d_3d_4d_5d_7^2 + 432d_2d_3d_5^2d_7^2 + 432d_2d_4d_5^2d_7^2 \\
+& 432d_3d_4d_5^2d_7^2 + 216d_2d_3d_4d_7^3 + 63d_2d_3d_5d_7^3 + 63d_2d_4d_5d_7^3 + 63d_3d_4d_5d_7^3,\\
B_2 =& 18(d_2d_3d_4d_5^2d_7^2 + 12d_2d_3d_4d_5^2d_7 - 7d_2d_3d_4d_5d_7^2 - 5d_2d_3d_5^2d_7^2 - 5d_2d_4d_5^2d_7^2 - 5d_3d_4d_5^2d_7^2 \\
-& 405d_2d_3d_4d_5^2 - 471d_2d_3d_4d_5d_7
- 123d_2d_3d_5^2d_7 - 123d_2d_4d_5^2d_7 - 123d_3d_4d_5^2d_7 - 96d_2d_3d_4d_7^2\\
 -& 33d_2d_3d_5d_7^2
- 33d_2d_4d_5d_7^2 - 33d_3d_4d_5d_7^2)d_5d_7,\\
C_2 =& -729(d_2d_3d_4d_5d_7 + 27d_2d_3d_4d_5 + 8d_2d_3d_4d_7 + d_2d_3d_5d_7 + d_2d_4d_5d_7 + d_3d_4d_5d_7)d_5^2d_7^2.
\end{align*}
\end{proposition}
\begin{proof}
If a corner of $T_3^S$ has a negative eigenvalue then so is $T_3^S$. Let $N$ be the $7 \times 7$ top left corner of $T_3^S$. By Lemma \ref{lem:1} and the fact that $d_i \ge 1$ for all $i$, the sign of $\det(N)$ is the sign of $A_2d_6^2 + B_2d_6 + C_2$, so if $Q_{k,-}<d_6<Q_{k,+}$, with $k=2$, then $N$ has a negative eigenvalue. But Lemma \ref{lem:1} states also that for $k=1$. Finally, observe that $Q_{k,-} < 0$ for $k=1,2$. The result follows.
\end{proof}

\begin{lemma} \label{lem:1} Let $N$ be the $7 \times 7$ top left corner of $T_3^S$. Then
$$\det(N) = \frac{(A_1d_6^2 + B_1d_6 + C_1)^2(A_2d_6^2 + B_2d_6 + C_2)}{d_2d_3d_4d_5^7d_6^6d_7^7},$$
with the notations of Proposition \ref{prop:1}.
Moreover, $N$ has an eigenvalue which is negative for $Q_{1,-}<d_6<Q_{1,+}$.

\end{lemma}
\begin{proof} See the SageMath computation in Appendix \ref{sub:1}.
\end{proof}

\begin{remark}
Observe that $\displaystyle Q_{1,+} < \frac{9}{2}(1+\sqrt{37}) = 31.87243\dots$ and $Q_{2,+} < 9(\sqrt{10}-1) = 19.4604989\dots $.
\end{remark}

\begin{remark}
Note that $\dim = [1,5,5,5,6,7,7]$ is covered by $Q_{2,+}$ only, because $Q_{2,+}(5,5,5,6,7) = 8.882676\dots > 7$, whereas $Q_{1,+}(5,5,5,6,7) = 6.537671\dots < 7$.
\end{remark}

\begin{remark} \label{rk:1}
The converse of Proposition \ref{prop:1} is false, because for $(d_2,d_3,d_4,d_5,d_7) = (5,5,5,6,7)$, the matrix $T_3^S$ has a negative eigenvalue if $d_6<\alpha = 9.5934\dots$, see Appendix \ref{sub:2}.
\end{remark}

\begin{remark} \label{rk:2}
To show how an ultimate update of Proposition \ref{prop:1} should be complex to state, note that $\alpha$ (in Remark \ref{rk:1}) is a root of an irreducible polynomial of degree $4$ in $\mathbb{Z}[X]$. It can be written as follows (see Appendix \ref{sub:3}):
$$\alpha = \frac{B}{a}+ \frac{1}{2}\sqrt{-A - i/j/A - k/b/B+ l/b} + \frac{m}{n}$$
where
$$
A=\left(\frac{f + c\sqrt{e}I}{d}\right)^{1/3}, \ B=\sqrt{bA + g/A + h},$$
and
\begin{align*}
a&=421113102, \
b=44334061169015601, \
c=250462485504, \
d=345734333761887148583413,\\
e&=5213988190704773354123819324759655, \
f=64964979666121194605838007808, \\
g&=149108026701745210176, \
h=3010843992410706004, \
i=16567558522416134464,\\
j&=4926006796557289, \
k=1151178558485955841275891856, \
l=6021687984821412008,\\
m&=241694327, \
n=210556551.
\end{align*}
In fact, $\alpha = Q(5,5,5,6,7)$ where $Q$ is an algebraic function requiring about one million characters (or about 500 pages) to be written (at least before any simplification). A laptop needed about 24 hours to compute it. Of course, we can compute $\alpha$ directly, see Appendix. In addition, it seems that $Q < \beta = \frac{9}{2}(1+\sqrt{37})$. So if $d_6 \ge \beta$ then $T_3^S$ should not have a negative eigenvalue.
\end{remark}

A natural question is that whether we could exclude fusion rings by the restriction condition on PFdim in Proposition \ref{prop:1}.
The answer is yes.
 With the help of computers, we find another
 9 simple (non-integral) fusion rings of rank 7 and multiplicity $\leq$ 6, having the same local data for $S=\{x_6,x_7\}$, as above, and whose exclusion from unitary categorification is covered by Proposition \ref{prop:1}.
\begin{align*}
\begin{smallmatrix}
1&0&0&0&0&0&0\\
0&1&0&0&0&0&0\\
0&0&1&0&0&0&0\\
0&0&0&1&0&0&0\\
0&0&0&0&1&0&0\\
0&0&0&0&0&1&0\\
0&0&0&0&0&0&1
\end{smallmatrix},
\begin{smallmatrix}
0&1&0&0&0&0&0\\
1&1&1&0&0&1&1\\
0&1&0&0&1&1&1\\
0&0&0&1&1&1&1\\
0&0&1&1&1&1&1\\
0&1&1&1&1&1&1\\
0&1&1&1&1&1&1
\end{smallmatrix},
\begin{smallmatrix}
0&0&1&0&0&0&0\\
0&1&0&0&1&1&1\\
1&0&1&1&0&1&1\\
0&0&1&0&1&1&1\\
0&1&0&1&1&1&1\\
0&1&1&1&1&1&1\\
0&1&1&1&1&1&1
\end{smallmatrix},
\begin{smallmatrix}
0&0&0&1&0&0&0\\
0&0&0&1&1&1&1\\
0&0&1&0&1&1&1\\
1&1&0&1&0&1&1\\
0&1&1&0&1&1&1\\
0&1&1&1&1&1&1\\
0&1&1&1&1&1&1
\end{smallmatrix},
\begin{smallmatrix}
0&0&0&0&1&0&0\\
0&0&1&1&1&1&1\\
0&1&0&1&1&1&1\\
0&1&1&0&1&1&1\\
1&1&1&1&1&1&1\\
0&1&1&1&1&2&1\\
0&1&1&1&1&1&2
\end{smallmatrix},
\begin{smallmatrix}
0&0&0&0&0&1&0\\
0&1&1&1&1&1&1\\
0&1&1&1&1&1&1\\
0&1&1&1&1&1&1\\
0&1&1&1&1&2&1\\
1&1&1&1&2&0&3\\
0&1&1&1&1&3&1
\end{smallmatrix},
\begin{smallmatrix}
0&0&0&0&0&0&1\\
0&1&1&1&1&1&1\\
0&1&1&1&1&1&1\\
0&1&1&1&1&1&1\\
0&1&1&1&1&1&2\\
0&1&1&1&1&3&1\\
1&1&1&1&2&1&2
\end{smallmatrix}
\end{align*}
\begin{align*}
\begin{smallmatrix}
1&0&0&0&0&0&0\\
0&1&0&0&0&0&0\\
0&0&1&0&0&0&0\\
0&0&0&1&0&0&0\\
0&0&0&0&1&0&0\\
0&0&0&0&0&1&0\\
0&0&0&0&0&0&1
\end{smallmatrix},
\begin{smallmatrix}
0&1&0&0&0&0&0\\
1&0&0&0&1&2&2\\
0&0&1&1&1&1&1\\
0&0&1&1&1&1&1\\
0&1&1&1&1&1&1\\
0&2&1&1&1&1&1\\
0&2&1&1&1&1&1
\end{smallmatrix},
\begin{smallmatrix}
0&0&1&0&0&0&0\\
0&0&1&1&1&1&1\\
1&1&1&0&0&1&1\\
0&1&0&0&1&1&1\\
0&1&0&1&1&1&1\\
0&1&1&1&1&1&1\\
0&1&1&1&1&1&1
\end{smallmatrix},
\begin{smallmatrix}
0&0&0&1&0&0&0\\
0&0&1&1&1&1&1\\
0&1&0&0&1&1&1\\
1&1&0&1&0&1&1\\
0&1&1&0&1&1&1\\
0&1&1&1&1&1&1\\
0&1&1&1&1&1&1
\end{smallmatrix},
\begin{smallmatrix}
0&0&0&0&1&0&0\\
0&1&1&1&1&1&1\\
0&1&0&1&1&1&1\\
0&1&1&0&1&1&1\\
1&1&1&1&1&1&1\\
0&1&1&1&1&2&1\\
0&1&1&1&1&1&2
\end{smallmatrix},
\begin{smallmatrix}
0&0&0&0&0&1&0\\
0&2&1&1&1&1&1\\
0&1&1&1&1&1&1\\
0&1&1&1&1&1&1\\
0&1&1&1&1&2&1\\
1&1&1&1&2&0&3\\
0&1&1&1&1&3&1
\end{smallmatrix},
\begin{smallmatrix}
0&0&0&0&0&0&1\\
0&2&1&1&1&1&1\\
0&1&1&1&1&1&1\\
0&1&1&1&1&1&1\\
0&1&1&1&1&1&2\\
0&1&1&1&1&3&1\\
1&1&1&1&2&1&2
\end{smallmatrix}
\end{align*}
\begin{align*}
\begin{smallmatrix}
1&0&0&0&0&0&0\\
0&1&0&0&0&0&0\\
0&0&1&0&0&0&0\\
0&0&0&1&0&0&0\\
0&0&0&0&1&0&0\\
0&0&0&0&0&1&0\\
0&0&0&0&0&0&1
\end{smallmatrix},
\begin{smallmatrix}
0&1&0&0&0&0&0\\
1&1&1&1&2&3&3\\
0&1&1&2&1&1&1\\
0&1&2&1&2&2&2\\
0&2&1&2&1&1&1\\
0&3&1&2&1&1&1\\
0&3&1&2&1&1&1
\end{smallmatrix},
\begin{smallmatrix}
0&0&1&0&0&0&0\\
0&1&1&2&1&1&1\\
1&1&0&1&0&1&1\\
0&2&1&0&1&1&1\\
0&1&0&1&1&1&1\\
0&1&1&1&1&1&1\\
0&1&1&1&1&1&1
\end{smallmatrix},
\begin{smallmatrix}
0&0&0&1&0&0&0\\
0&1&2&1&2&2&2\\
0&2&1&0&1&1&1\\
1&1&0&2&1&2&2\\
0&2&1&1&1&1&1\\
0&2&1&2&1&1&1\\
0&2&1&2&1&1&1
\end{smallmatrix},
\begin{smallmatrix}
0&0&0&0&1&0&0\\
0&2&1&2&1&1&1\\
0&1&0&1&1&1&1\\
0&2&1&1&1&1&1\\
1&1&1&1&1&1&1\\
0&1&1&1&1&2&1\\
0&1&1&1&1&1&2
\end{smallmatrix},
\begin{smallmatrix}
0&0&0&0&0&1&0\\
0&3&1&2&1&1&1\\
0&1&1&1&1&1&1\\
0&2&1&2&1&1&1\\
0&1&1&1&1&2&1\\
1&1&1&1&2&0&3\\
0&1&1&1&1&3&1
\end{smallmatrix},
\begin{smallmatrix}
0&0&0&0&0&0&1\\
0&3&1&2&1&1&1\\
0&1&1&1&1&1&1\\
0&2&1&2&1&1&1\\
0&1&1&1&1&1&2\\
0&1&1&1&1&3&1\\
1&1&1&1&2&1&2
\end{smallmatrix}
\end{align*}
\begin{align*}
\begin{smallmatrix}
1&0&0&0&0&0&0\\
0&1&0&0&0&0&0\\
0&0&1&0&0&0&0\\
0&0&0&1&0&0&0\\
0&0&0&0&1&0&0\\
0&0&0&0&0&1&0\\
0&0&0&0&0&0&1
\end{smallmatrix},
\begin{smallmatrix}
0&1&0&0&0&0&0\\
1&4&2&4&3&4&4\\
0&2&1&3&1&1&1\\
0&4&3&3&3&3&3\\
0&3&1&3&1&1&1\\
0&4&1&3&1&1&1\\
0&4&1&3&1&1&1
\end{smallmatrix},
\begin{smallmatrix}
0&0&1&0&0&0&0\\
0&2&1&3&1&1&1\\
1&1&0&1&0&1&1\\
0&3&1&1&1&1&1\\
0&1&0&1&1&1&1\\
0&1&1&1&1&1&1\\
0&1&1&1&1&1&1
\end{smallmatrix},
\begin{smallmatrix}
0&0&0&1&0&0&0\\
0&4&3&3&3&3&3\\
0&3&1&1&1&1&1\\
1&3&1&4&2&3&3\\
0&3&1&2&1&1&1\\
0&3&1&3&1&1&1\\
0&3&1&3&1&1&1
\end{smallmatrix},
\begin{smallmatrix}
0&0&0&0&1&0&0\\
0&3&1&3&1&1&1\\
0&1&0&1&1&1&1\\
0&3&1&2&1&1&1\\
1&1&1&1&1&1&1\\
0&1&1&1&1&2&1\\
0&1&1&1&1&1&2
\end{smallmatrix},
\begin{smallmatrix}
0&0&0&0&0&1&0\\
0&4&1&3&1&1&1\\
0&1&1&1&1&1&1\\
0&3&1&3&1&1&1\\
0&1&1&1&1&2&1\\
1&1&1&1&2&0&3\\
0&1&1&1&1&3&1
\end{smallmatrix},
\begin{smallmatrix}
0&0&0&0&0&0&1\\
0&4&1&3&1&1&1\\
0&1&1&1&1&1&1\\
0&3&1&3&1&1&1\\
0&1&1&1&1&1&2\\
0&1&1&1&1&3&1\\
1&1&1&1&2&1&2
\end{smallmatrix}
\end{align*}
\begin{align*}
\begin{smallmatrix}
1&0&0&0&0&0&0\\
0&1&0&0&0&0&0\\
0&0&1&0&0&0&0\\
0&0&0&1&0&0&0\\
0&0&0&0&1&0&0\\
0&0&0&0&0&1&0\\
0&0&0&0&0&0&1
\end{smallmatrix},
\begin{smallmatrix}
0&1&0&0&0&0&0\\
1&1&1&0&0&1&1\\
0&1&3&1&1&1&1\\
0&0&1&0&1&1&1\\
0&0&1&1&1&1&1\\
0&1&1&1&1&1&1\\
0&1&1&1&1&1&1
\end{smallmatrix},
\begin{smallmatrix}
0&0&1&0&0&0&0\\
0&1&3&1&1&1&1\\
1&3&0&3&4&5&5\\
0&1&3&1&1&1&1\\
0&1&4&1&1&1&1\\
0&1&5&1&1&1&1\\
0&1&5&1&1&1&1
\end{smallmatrix},
\begin{smallmatrix}
0&0&0&1&0&0&0\\
0&0&1&0&1&1&1\\
0&1&3&1&1&1&1\\
1&0&1&1&0&1&1\\
0&1&1&0&1&1&1\\
0&1&1&1&1&1&1\\
0&1&1&1&1&1&1
\end{smallmatrix},
\begin{smallmatrix}
0&0&0&0&1&0&0\\
0&0&1&1&1&1&1\\
0&1&4&1&1&1&1\\
0&1&1&0&1&1&1\\
1&1&1&1&1&1&1\\
0&1&1&1&1&2&1\\
0&1&1&1&1&1&2
\end{smallmatrix},
\begin{smallmatrix}
0&0&0&0&0&1&0\\
0&1&1&1&1&1&1\\
0&1&5&1&1&1&1\\
0&1&1&1&1&1&1\\
0&1&1&1&1&2&1\\
1&1&1&1&2&0&3\\
0&1&1&1&1&3&1
\end{smallmatrix},
\begin{smallmatrix}
0&0&0&0&0&0&1\\
0&1&1&1&1&1&1\\
0&1&5&1&1&1&1\\
0&1&1&1&1&1&1\\
0&1&1&1&1&1&2\\
0&1&1&1&1&3&1\\
1&1&1&1&2&1&2
\end{smallmatrix}
\end{align*}
\begin{align*}
\begin{smallmatrix}
1&0&0&0&0&0&0\\
0&1&0&0&0&0&0\\
0&0&1&0&0&0&0\\
0&0&0&1&0&0&0\\
0&0&0&0&1&0&0\\
0&0&0&0&0&1&0\\
0&0&0&0&0&0&1
\end{smallmatrix},
\begin{smallmatrix}
0&1&0&0&0&0&0\\
1&0&1&1&0&1&1\\
0&1&3&2&1&1&1\\
0&1&2&0&1&1&1\\
0&0&1&1&1&1&1\\
0&1&1&1&1&1&1\\
0&1&1&1&1&1&1
\end{smallmatrix},
\begin{smallmatrix}
0&0&1&0&0&0&0\\
0&1&3&2&1&1&1\\
1&3&3&3&4&5&5\\
0&2&3&3&2&2&2\\
0&1&4&2&1&1&1\\
0&1&5&2&1&1&1\\
0&1&5&2&1&1&1
\end{smallmatrix},
\begin{smallmatrix}
0&0&0&1&0&0&0\\
0&1&2&0&1&1&1\\
0&2&3&3&2&2&2\\
1&0&3&0&1&2&2\\
0&1&2&1&1&1&1\\
0&1&2&2&1&1&1\\
0&1&2&2&1&1&1
\end{smallmatrix},
\begin{smallmatrix}
0&0&0&0&1&0&0\\
0&0&1&1&1&1&1\\
0&1&4&2&1&1&1\\
0&1&2&1&1&1&1\\
1&1&1&1&1&1&1\\
0&1&1&1&1&2&1\\
0&1&1&1&1&1&2
\end{smallmatrix},
\begin{smallmatrix}
0&0&0&0&0&1&0\\
0&1&1&1&1&1&1\\
0&1&5&2&1&1&1\\
0&1&2&2&1&1&1\\
0&1&1&1&1&2&1\\
1&1&1&1&2&0&3\\
0&1&1&1&1&3&1
\end{smallmatrix},
\begin{smallmatrix}
0&0&0&0&0&0&1\\
0&1&1&1&1&1&1\\
0&1&5&2&1&1&1\\
0&1&2&2&1&1&1\\
0&1&1&1&1&1&2\\
0&1&1&1&1&3&1\\
1&1&1&1&2&1&2
\end{smallmatrix}
\end{align*}
\begin{align*}
\begin{smallmatrix}
1&0&0&0&0&0&0\\
0&1&0&0&0&0&0\\
0&0&1&0&0&0&0\\
0&0&0&1&0&0&0\\
0&0&0&0&1&0&0\\
0&0&0&0&0&1&0\\
0&0&0&0&0&0&1
\end{smallmatrix},
\begin{smallmatrix}
0&1&0&0&0&0&0\\
1&5&5&3&3&4&4\\
0&5&3&4&3&3&3\\
0&3&4&0&2&2&2\\
0&3&3&2&1&1&1\\
0&4&3&2&1&1&1\\
0&4&3&2&1&1&1
\end{smallmatrix},
\begin{smallmatrix}
0&0&1&0&0&0&0\\
0&5&3&4&3&3&3\\
1&3&1&5&3&4&4\\
0&4&5&0&0&0&0\\
0&3&3&0&1&1&1\\
0&3&4&0&1&1&1\\
0&3&4&0&1&1&1
\end{smallmatrix},
\begin{smallmatrix}
0&0&0&1&0&0&0\\
0&3&4&0&2&2&2\\
0&4&5&0&0&0&0\\
1&0&0&2&2&3&3\\
0&2&0&2&1&1&1\\
0&2&0&3&1&1&1\\
0&2&0&3&1&1&1
\end{smallmatrix},
\begin{smallmatrix}
0&0&0&0&1&0&0\\
0&3&3&2&1&1&1\\
0&3&3&0&1&1&1\\
0&2&0&2&1&1&1\\
1&1&1&1&1&1&1\\
0&1&1&1&1&2&1\\
0&1&1&1&1&1&2
\end{smallmatrix},
\begin{smallmatrix}
0&0&0&0&0&1&0\\
0&4&3&2&1&1&1\\
0&3&4&0&1&1&1\\
0&2&0&3&1&1&1\\
0&1&1&1&1&2&1\\
1&1&1&1&2&0&3\\
0&1&1&1&1&3&1
\end{smallmatrix},
\begin{smallmatrix}
0&0&0&0&0&0&1\\
0&4&3&2&1&1&1\\
0&3&4&0&1&1&1\\
0&2&0&3&1&1&1\\
0&1&1&1&1&1&2\\
0&1&1&1&1&3&1\\
1&1&1&1&2&1&2
\end{smallmatrix}
\end{align*}
\begin{align*}
\begin{smallmatrix}
1&0&0&0&0&0&0\\
0&1&0&0&0&0&0\\
0&0&1&0&0&0&0\\
0&0&0&1&0&0&0\\
0&0&0&0&1&0&0\\
0&0&0&0&0&1&0\\
0&0&0&0&0&0&1
\end{smallmatrix},
\begin{smallmatrix}
0&1&0&0&0&0&0\\
1&1&1&0&0&1&1\\
0&1&4&1&1&1&1\\
0&0&1&0&1&1&1\\
0&0&1&1&1&1&1\\
0&1&1&1&1&1&1\\
0&1&1&1&1&1&1
\end{smallmatrix},
\begin{smallmatrix}
0&0&1&0&0&0&0\\
0&1&4&1&1&1&1\\
1&4&4&4&5&6&6\\
0&1&4&1&1&1&1\\
0&1&5&1&1&1&1\\
0&1&6&1&1&1&1\\
0&1&6&1&1&1&1
\end{smallmatrix},
\begin{smallmatrix}
0&0&0&1&0&0&0\\
0&0&1&0&1&1&1\\
0&1&4&1&1&1&1\\
1&0&1&1&0&1&1\\
0&1&1&0&1&1&1\\
0&1&1&1&1&1&1\\
0&1&1&1&1&1&1
\end{smallmatrix},
\begin{smallmatrix}
0&0&0&0&1&0&0\\
0&0&1&1&1&1&1\\
0&1&5&1&1&1&1\\
0&1&1&0&1&1&1\\
1&1&1&1&1&1&1\\
0&1&1&1&1&2&1\\
0&1&1&1&1&1&2
\end{smallmatrix},
\begin{smallmatrix}
0&0&0&0&0&1&0\\
0&1&1&1&1&1&1\\
0&1&6&1&1&1&1\\
0&1&1&1&1&1&1\\
0&1&1&1&1&2&1\\
1&1&1&1&2&0&3\\
0&1&1&1&1&3&1
\end{smallmatrix},
\begin{smallmatrix}
0&0&0&0&0&0&1\\
0&1&1&1&1&1&1\\
0&1&6&1&1&1&1\\
0&1&1&1&1&1&1\\
0&1&1&1&1&1&2\\
0&1&1&1&1&3&1\\
1&1&1&1&2&1&2
\end{smallmatrix}
\end{align*}
\begin{align*}
\begin{smallmatrix}
1&0&0&0&0&0&0\\
0&1&0&0&0&0&0\\
0&0&1&0&0&0&0\\
0&0&0&1&0&0&0\\
0&0&0&0&1&0&0\\
0&0&0&0&0&1&0\\
0&0&0&0&0&0&1
\end{smallmatrix},
\begin{smallmatrix}
0&1&0&0&0&0&0\\
1&1&5&4&3&4&4\\
0&5&2&5&1&1&1\\
0&4&5&2&3&3&3\\
0&3&1&3&1&1&1\\
0&4&1&3&1&1&1\\
0&4&1&3&1&1&1
\end{smallmatrix},
\begin{smallmatrix}
0&0&1&0&0&0&0\\
0&5&2&5&1&1&1\\
1&2&1&2&3&4&4\\
0&5&2&3&2&2&2\\
0&1&3&2&1&1&1\\
0&1&4&2&1&1&1\\
0&1&4&2&1&1&1
\end{smallmatrix},
\begin{smallmatrix}
0&0&0&1&0&0&0\\
0&4&5&2&3&3&3\\
0&5&2&3&2&2&2\\
1&2&3&6&2&3&3\\
0&3&2&2&1&1&1\\
0&3&2&3&1&1&1\\
0&3&2&3&1&1&1
\end{smallmatrix},
\begin{smallmatrix}
0&0&0&0&1&0&0\\
0&3&1&3&1&1&1\\
0&1&3&2&1&1&1\\
0&3&2&2&1&1&1\\
1&1&1&1&1&1&1\\
0&1&1&1&1&2&1\\
0&1&1&1&1&1&2
\end{smallmatrix},
\begin{smallmatrix}
0&0&0&0&0&1&0\\
0&4&1&3&1&1&1\\
0&1&4&2&1&1&1\\
0&3&2&3&1&1&1\\
0&1&1&1&1&2&1\\
1&1&1&1&2&0&3\\
0&1&1&1&1&3&1
\end{smallmatrix},
\begin{smallmatrix}
0&0&0&0&0&0&1\\
0&4&1&3&1&1&1\\
0&1&4&2&1&1&1\\
0&3&2&3&1&1&1\\
0&1&1&1&1&1&2\\
0&1&1&1&1&3&1\\
1&1&1&1&2&1&2
\end{smallmatrix}
\end{align*}
\noindent 
If we release the bound of the multiplicities, then there may exist an infinite family of fusion rings with the same local data that could be excluded.
\subsubsection{An Orbifold Construction of Fusion Rings}
Considering the orbifold theory, we exclude a family of fusion rings with unbounded Frobenius-Perron dimension. 
\begin{definition}\label{def:orbifold}
Suppose $\mathcal{R}$ is a fusion ring with basis $\{x_1=1,x_2,\ldots,x_m\}$  and $M_i$ is the fusion matrix of $x_i$. Let $\mathcal{R}_\ell$ be the fusion ring obtained by a $\mathbb{Z}_\ell$ action on $\mathcal{R}\otimes \mathbb{Z}_\ell$ with basis $\{x_{ij},X_k:\ 0\leq i,j\leq \ell-1, 2\leq k\leq m\}$, where $x_{ij}=g^i\otimes g^j$ and $X_k=\{x_k\otimes I,x_k\otimes g,\ldots,x_k\otimes g^{n-1} \}$, $g$ is the generator of $\mathbb{Z}_\ell$.
\end{definition}
\begin{proposition}\label{prop:orbifold theorem}
Suppose $\mathcal{R}$ is a fusion ring with basis $\{x_1=1,x_2,\ldots,x_m\}$ and $\mathcal{R}_\ell$ is the fusion ring obtained in Definition \ref{def:orbifold}.
Then for any local set $S\subseteq\{X_k:2\leq k\leq m  \}$, we have
\begin{align*}
T_n^S(\mathcal{R}_\ell)=\ell^2 T_n^S(\mathcal{R}),
\end{align*}
where $T_n^S(\mathcal{R}_\ell)$ and $T_n^S(\mathcal{R})$ are the primary $(S,n)$-matrices of $\mathcal{R}_\ell$ and $\mathcal{R}$ respectively.
\end{proposition}
\begin{proof}
Note that $M_{x_{ij}}^S=I$ and $M_{X_k}^S=\ell M_i^S$, $\|M_{x_{ij}}\|=1$ and $\|M_{X_k}\|=\ell\| M_i\|$. So
\begin{align*}
T_n^S(\mathcal{R}_\ell)&=\sum_{0\leq i,j\leq \ell-1}
\|M_{x_{ij}}\|^2\bigg(\frac{M_{x_{ij}}^S}{\|M_{x_{ij}}\|}\bigg)^{\otimes n}+\sum_{k=2}^m \|M_{X_k}\|^2\bigg(  \frac{M_{X_k}^S}{\|M_{X_k}\|} \bigg)^{\otimes n}\\
&=\ell^2 M_1+\ell^2\sum_{k=2}^m \|M_k\|^2\bigg(  \frac{M_{k}^S}{\|M_{k}\|} \bigg)^{\otimes n}\\
&=\ell^2T_n^S(\mathcal{R}).
\end{align*}
This completes the proof of the theorem.
\end{proof}

\begin{remark}\label{rem:orbifold}
Let $\mathcal{R}$ be the above simple integral fusion ring $\mathcal{K}_7$ of rank 7 and $\mathcal{R}_\ell$ is the fusion ring obtained from $\mathcal{R}$ in Proposition \ref{prop:orbifold theorem}, $\ell\geq1$. So $T_3^S(\mathcal{R}_\ell)=\ell^2 T_3^S(\mathcal{R})$ is not positive since it has a negative eigenvalue $\simeq
-0.62949\ell^2$.
 Therefore,  $\mathcal{R}_\ell$ admits no unitary categorification by Theorem \ref{thm:localization criterion}. We obtain a family of fusion rings with unbounded dimension, which admit no unitary categorification.
\end{remark}
\subsection{Five Simple Fusion Rings Excluded in a Single Time} \label{sub:S8}
There are  five simple integral fusion rings of rank 8, FPdim 660, type [[1,1],[5,2],[10,2],[11,1],[12,2]], multiplicity 5 on page 57 in \cite{LPW21}, with the same local data on the last two simple objects.
\begin{align*}\footnotesize
\begin{smallmatrix}
1&0&0&0&0&0&0&0\\
0&1&0&0&0&0&0&0\\
0&0&1&0&0&0&0&0\\
0&0&0&1&0&0&0&0\\
0&0&0&0&1&0&0&0\\
0&0&0&0&0&1&0&0\\
0&0&0&0&0&0&1&0\\
0&0&0&0&0&0&0&1\\
\end{smallmatrix},
\begin{smallmatrix}
0&1&0&0&0&0&0&0\\
0&0&1&1&1&0&0&0\\
1&0&0&0&0&0&1&1\\
0&0&1&0&1&1&1&1\\
0&0&1&1&0&1&1&1\\
0&0&0&1&1&1&1&1\\
0&1&0&1&1&1&1&1\\
0&1&0&1&1&1&1&1
\end{smallmatrix},
\begin{smallmatrix}
0&0&1&0&0&0&0&0\\
1&0&0&0&0&0&1&1\\
0&1&0&1&1&0&0&0\\
0&1&0&0&1&1&1&1\\
0&1&0&1&0&1&1&1\\
0&0&0&1&1&1&1&1\\
0&0&1&1&1&1&1&1\\
0&0&1&1&1&1&1&1
\end{smallmatrix},
\begin{smallmatrix}
0&0&0&1&0&0&0&0\\
0&0&1&0&1&1&1&1\\
0&1&0&0&1&1&1&1\\
1&0&0&2&2&1&2&2\\
0&1&1&2&0&2&2&2\\
0&1&1&1&2&2&2&2\\
0&1&1&2&2&2&2&2\\
0&1&1&2&2&2&2&2
\end{smallmatrix},
\begin{smallmatrix}
0&0&0&0&1&0&0&0\\
0&0&1&1&0&1&1&1\\
0&1&0&1&0&1&1&1\\
0&1&1&2&0&2&2&2\\
1&0&0&0&4&1&2&2\\
0&1&1&2&1&2&2&2\\
0&1&1&2&2&2&2&2\\
0&1&1&2&2&2&2&2
\end{smallmatrix},
\begin{smallmatrix}
0&0&0&0&0&1&0&0\\
0&0&0&1&1&1&1&1\\
0&0&0&1&1&1&1&1\\
0&1&1&1&2&2&2&2\\
0&1&1&2&1&2&2&2\\
1&1&1&2&2&2&2&2\\
0&1&1&2&2&2&3&2\\
0&1&1&2&2&2&2&3
\end{smallmatrix},
\begin{smallmatrix}
0&0&0&0&0&0&1&0\\
0&1&0&1&1&1&1&1\\
0&0&1&1&1&1&1&1\\
0&1&1&2&2&2&2&2\\
0&1&1&2&2&2&2&2\\
0&1&1&2&2&2&3&2\\
1&1&1&2&2&3&0&5\\
0&1&1&2&2&2&5&1
\end{smallmatrix},
\begin{smallmatrix}
0&0&0&0&0&0&0&1\\
0&1&0&1&1&1&1&1\\
0&0&1&1&1&1&1&1\\
0&1&1&2&2&2&2&2\\
0&1&1&2&2&2&2&2\\
0&1&1&2&2&2&2&3\\
0&1&1&2&2&2&5&1\\
1&1&1&2&2&3&1&4
\end{smallmatrix}
\end{align*}
\begin{align*}\footnotesize
\begin{smallmatrix}
1&0&0&0&0&0&0&0\\
0&1&0&0&0&0&0&0\\
0&0&1&0&0&0&0&0\\
0&0&0&1&0&0&0&0\\
0&0&0&0&1&0&0&0\\
0&0&0&0&0&1&0&0\\
0&0&0&0&0&0&1&0\\
0&0&0&0&0&0&0&1\\
\end{smallmatrix},
\begin{smallmatrix}
0&1&0&0&0&0&0&0\\
0&0&1&1&1&0&0&0\\
1&0&0&0&0&0&1&1\\
0&0&1&0&1&1&1&1\\
0&0&1&1&0&1&1&1\\
0&0&0&1&1&1&1&1\\
0&1&0&1&1&1&1&1\\
0&1&0&1&1&1&1&1
\end{smallmatrix},
\begin{smallmatrix}
0&0&1&0&0&0&0&0\\
1&0&0&0&0&0&1&1\\
0&1&0&1&1&0&0&0\\
0&1&0&0&1&1&1&1\\
0&1&0&1&0&1&1&1\\
0&0&0&1&1&1&1&1\\
0&0&1&1&1&1&1&1\\
0&0&1&1&1&1&1&1
\end{smallmatrix},
\begin{smallmatrix}
0&0&0&1&0&0&0&0\\
0&0&1&0&1&1&1&1\\
0&1&0&0&1&1&1&1\\
1&0&0&3&1&1&2&2\\
0&1&1&1&1&2&2&2\\
0&1&1&1&2&2&2&2\\
0&1&1&2&2&2&2&2\\
0&1&1&2&2&2&2&2
\end{smallmatrix},
\begin{smallmatrix}
0&0&0&0&1&0&0&0\\
0&0&1&1&0&1&1&1\\
0&1&0&1&0&1&1&1\\
0&1&1&1&1&2&2&2\\
1&0&0&1&3&1&2&2\\
0&1&1&2&1&2&2&2\\
0&1&1&2&2&2&2&2\\
0&1&1&2&2&2&2&2
\end{smallmatrix},
\begin{smallmatrix}
0&0&0&0&0&1&0&0\\
0&0&0&1&1&1&1&1\\
0&0&0&1&1&1&1&1\\
0&1&1&1&2&2&2&2\\
0&1&1&2&1&2&2&2\\
1&1&1&2&2&2&2&2\\
0&1&1&2&2&2&3&2\\
0&1&1&2&2&2&2&3
\end{smallmatrix},
\begin{smallmatrix}
0&0&0&0&0&0&1&0\\
0&1&0&1&1&1&1&1\\
0&0&1&1&1&1&1&1\\
0&1&1&2&2&2&2&2\\
0&1&1&2&2&2&2&2\\
0&1&1&2&2&2&3&2\\
1&1&1&2&2&3&0&5\\
0&1&1&2&2&2&5&1
\end{smallmatrix},
\begin{smallmatrix}
0&0&0&0&0&0&0&1\\
0&1&0&1&1&1&1&1\\
0&0&1&1&1&1&1&1\\
0&1&1&2&2&2&2&2\\
0&1&1&2&2&2&2&2\\
0&1&1&2&2&2&2&3\\
0&1&1&2&2&2&5&1\\
1&1&1&2&2&3&1&4
\end{smallmatrix}
\end{align*}
\begin{align*}\footnotesize
\begin{smallmatrix}
1&0&0&0&0&0&0&0\\
0&1&0&0&0&0&0&0\\
0&0&1&0&0&0&0&0\\
0&0&0&1&0&0&0&0\\
0&0&0&0&1&0&0&0\\
0&0&0&0&0&1&0&0\\
0&0&0&0&0&0&1&0\\
0&0&0&0&0&0&0&1\\
\end{smallmatrix},
\begin{smallmatrix}
0&1&0&0&0&0&0&0\\
0&0&1&1&1&0&0&0\\
1&0&0&0&0&0&1&1\\
0&0&1&1&0&1&1&1\\
0&0&1&0&1&1&1&1\\
0&0&0&1&1&1&1&1\\
0&1&0&1&1&1&1&1\\
0&1&0&1&1&1&1&1
\end{smallmatrix},
\begin{smallmatrix}
0&0&1&0&0&0&0&0\\
1&0&0&0&0&0&1&1\\
0&1&0&1&1&0&0&0\\
0&1&0&1&0&1&1&1\\
0&1&0&0&1&1&1&1\\
0&0&0&1&1&1&1&1\\
0&0&1&1&1&1&1&1\\
0&0&1&1&1&1&1&1
\end{smallmatrix},
\begin{smallmatrix}
0&0&0&1&0&0&0&0\\
0&0&1&1&0&1&1&1\\
0&1&0&1&0&1&1&1\\
1&1&1&0&3&1&2&2\\
0&0&0&3&0&2&2&2\\
0&1&1&1&2&2&2&2\\
0&1&1&2&2&2&2&2\\
0&1&1&2&2&2&2&2
\end{smallmatrix},
\begin{smallmatrix}
0&0&0&0&1&0&0&0\\
0&0&1&0&1&1&1&1\\
0&1&0&0&1&1&1&1\\
0&0&0&3&0&2&2&2\\
1&1&1&0&3&1&2&2\\
0&1&1&2&1&2&2&2\\
0&1&1&2&2&2&2&2\\
0&1&1&2&2&2&2&2
\end{smallmatrix},
\begin{smallmatrix}
0&0&0&0&0&1&0&0\\
0&0&0&1&1&1&1&1\\
0&0&0&1&1&1&1&1\\
0&1&1&1&2&2&2&2\\
0&1&1&2&1&2&2&2\\
1&1&1&2&2&2&2&2\\
0&1&1&2&2&2&3&2\\
0&1&1&2&2&2&2&3\\
\end{smallmatrix},
\begin{smallmatrix}
0&0&0&0&0&0&1&0\\
0&1&0&1&1&1&1&1\\
0&0&1&1&1&1&1&1\\
0&1&1&2&2&2&2&2\\
0&1&1&2&2&2&2&2\\
0&1&1&2&2&2&3&2\\
1&1&1&2&2&3&0&5\\
0&1&1&2&2&2&5&1
\end{smallmatrix},
\begin{smallmatrix}
0&0&0&0&0&0&0&1\\
0&1&0&1&1&1&1&1\\
0&0&1&1&1&1&1&1\\
0&1&1&2&2&2&2&2\\
0&1&1&2&2&2&2&2\\
0&1&1&2&2&2&2&3\\
0&1&1&2&2&2&5&1\\
1&1&1&2&2&3&1&4
\end{smallmatrix}
\end{align*}
\begin{align*}\footnotesize
\begin{smallmatrix}
1&0&0&0&0&0&0&0\\
0&1&0&0&0&0&0&0\\
0&0&1&0&0&0&0&0\\
0&0&0&1&0&0&0&0\\
0&0&0&0&1&0&0&0\\
0&0&0&0&0&1&0&0\\
0&0&0&0&0&0&1&0\\
0&0&0&0&0&0&0&1\\
\end{smallmatrix},
\begin{smallmatrix}
0&1&0&0&0&0&0&0\\
0&0&1&1&1&0&0&0\\
1&0&0&0&0&0&1&1\\
0&0&1&1&0&1&1&1\\
0&0&1&0&1&1&1&1\\
0&0&0&1&1&1&1&1\\
0&1&0&1&1&1&1&1\\
0&1&0&1&1&1&1&1
\end{smallmatrix},
\begin{smallmatrix}
0&0&1&0&0&0&0&0\\
1&0&0&0&0&0&1&1\\
0&1&0&1&1&0&0&0\\
0&1&0&1&0&1&1&1\\
0&1&0&0&1&1&1&1\\
0&0&0&1&1&1&1&1\\
0&0&1&1&1&1&1&1\\
0&0&1&1&1&1&1&1
\end{smallmatrix},
\begin{smallmatrix}
0&0&0&1&0&0&0&0\\
0&0&1&1&0&1&1&1\\
0&1&0&1&0&1&1&1\\
1&1&1&1&2&1&2&2\\
0&0&0&2&1&2&2&2\\
0&1&1&1&2&2&2&2\\
0&1&1&2&2&2&2&2\\
0&1&1&2&2&2&2&2
\end{smallmatrix},
\begin{smallmatrix}
0&0&0&0&1&0&0&0\\
0&0&1&0&1&1&1&1\\
0&1&0&0&1&1&1&1\\
0&0&0&2&1&2&2&2\\
1&1&1&1&2&1&2&2\\
0&1&1&2&1&2&2&2\\
0&1&1&2&2&2&2&2\\
0&1&1&2&2&2&2&2
\end{smallmatrix},
\begin{smallmatrix}
0&0&0&0&0&1&0&0\\
0&0&0&1&1&1&1&1\\
0&0&0&1&1&1&1&1\\
0&1&1&1&2&2&2&2\\
0&1&1&2&1&2&2&2\\
1&1&1&2&2&2&2&2\\
0&1&1&2&2&2&3&2\\
0&1&1&2&2&2&2&3
\end{smallmatrix},
\begin{smallmatrix}
0&0&0&0&0&0&1&0\\
0&1&0&1&1&1&1&1\\
0&0&1&1&1&1&1&1\\
0&1&1&2&2&2&2&2\\
0&1&1&2&2&2&2&2\\
0&1&1&2&2&2&3&2\\
1&1&1&2&2&3&0&5\\
0&1&1&2&2&2&5&1
\end{smallmatrix},
\begin{smallmatrix}
0&0&0&0&0&0&0&1\\
0&1&0&1&1&1&1&1\\
0&0&1&1&1&1&1&1\\
0&1&1&2&2&2&2&2\\
0&1&1&2&2&2&2&2\\
0&1&1&2&2&2&2&3\\
0&1&1&2&2&2&5&1\\
1&1&1&2&2&3&1&4
\end{smallmatrix}
\end{align*}
\begin{align*}\footnotesize
\begin{smallmatrix}
1&0&0&0&0&0&0&0\\
0&1&0&0&0&0&0&0\\
0&0&1&0&0&0&0&0\\
0&0&0&1&0&0&0&0\\
0&0&0&0&1&0&0&0\\
0&0&0&0&0&1&0&0\\
0&0&0&0&0&0&1&0\\
0&0&0&0&0&0&0&1\\
\end{smallmatrix},
\begin{smallmatrix}
0&1&0&0&0&0&0&0\\
0&0&1&1&1&0&0&0\\
1&0&0&0&0&0&1&1\\
0&0&1&0&1&1&1&1\\
0&0&1&1&0&1&1&1\\
0&0&0&1&1&1&1&1\\
0&1&0&1&1&1&1&1\\
0&1&0&1&1&1&1&1
\end{smallmatrix},
\begin{smallmatrix}
0&0&1&0&0&0&0&0\\
1&0&0&0&0&0&1&1\\
0&1&0&1&1&0&0&0\\
0&1&0&0&1&1&1&1\\
0&1&0&1&0&1&1&1\\
0&0&0&1&1&1&1&1\\
0&0&1&1&1&1&1&1\\
0&0&1&1&1&1&1&1
\end{smallmatrix},
\begin{smallmatrix}
0&0&0&1&0&0&0&0\\
0&0&1&0&1&1&1&1\\
0&1&0&0&1&1&1&1\\
0&1&1&2&0&2&2&2\\
1&0&0&2&2&1&2&2\\
0&1&1&1&2&2&2&2\\
0&1&1&2&2&2&2&2\\
0&1&1&2&2&2&2&2
\end{smallmatrix},
\begin{smallmatrix}
0&0&0&0&1&0&0&0\\
0&0&1&1&0&1&1&1\\
0&1&0&1&0&1&1&1\\
1&0&0&2&2&1&2&2\\
0&1&1&0&2&2&2&2\\
0&1&1&2&1&2&2&2\\
0&1&1&2&2&2&2&2\\
0&1&1&2&2&2&2&2
\end{smallmatrix},
\begin{smallmatrix}
0&0&0&0&0&1&0&0\\
0&0&0&1&1&1&1&1\\
0&0&0&1&1&1&1&1\\
0&1&1&1&2&2&2&2\\
0&1&1&2&1&2&2&2\\
1&1&1&2&2&2&2&2\\
0&1&1&2&2&2&3&2\\
0&1&1&2&2&2&2&3
\end{smallmatrix},
\begin{smallmatrix}
0&0&0&0&0&0&1&0\\
0&1&0&1&1&1&1&1\\
0&0&1&1&1&1&1&1\\
0&1&1&2&2&2&2&2\\
0&1&1&2&2&2&2&2\\
0&1&1&2&2&2&3&2\\
1&1&1&2&2&3&0&5\\
0&1&1&2&2&2&5&1
\end{smallmatrix},
\begin{smallmatrix}
0&0&0&0&0&0&0&1\\
0&1&0&1&1&1&1&1\\
0&0&1&1&1&1&1&1\\
0&1&1&2&2&2&2&2\\
0&1&1&2&2&2&2&2\\
0&1&1&2&2&2&2&3\\
0&1&1&2&2&2&5&1\\
1&1&1&2&2&3&1&4
\end{smallmatrix}
\end{align*}
They are all excluded by the Schur product criterion in \cite{LPW21}. Here, we exclude them again by localized criterion in a single time.
Let $S=\{x_7,x_8\}$. Then they have the same local fusion matrices:
\begin{align*}
M_1^S=\begin{pmatrix}
1&0\\0&1
\end{pmatrix},M_2^S=M_3^S=\begin{pmatrix}
1&1\\1&1
\end{pmatrix},M_4^S=M_5^S=\begin{pmatrix}
2&2\\2&2
\end{pmatrix}.
\end{align*}
\begin{align*}
M_6^S=\begin{pmatrix}
3&2\\2&3
\end{pmatrix},M_7^S=\begin{pmatrix}
0&5\\5&1
\end{pmatrix},
M_8^S=\begin{pmatrix}
5&1\\1&4
\end{pmatrix}.
\end{align*}
So the primary $(S,3)$-matrix $T_3^S$ of them are  same as follows:
\begin{align*}
\frac{1}{132}\begin{pmatrix}
2095  &755 & 755&  463 & 755  &463  &463& 1746\\
755 &1820 & 463&  700 & 463 & 700 &1746 & 727\\
 755&  463& 1820&  700 & 463& 1746  &700&  727\\
  463  &700 & 700 &1600& 1746 & 727 & 727  &711\\
 755  &463 & 463 &1746& 1820 & 700  &700 & 727\\
 463  &700& 1746 & 727&  700& 1600 & 727 & 711\\
  463 &1746  &700 & 727&  700 & 727 &1600 & 711\\
 1746 & 727&  727  &711&  727  &711&  711& 1435
\end{pmatrix}.
\end{align*}
It is not positive since it has a negative eigenvalue $\simeq -2.948812$ (see Appendix \ref{App:S8}). 
Hence these five fusion rings admit no unitary categorification by Theorem \ref{thm:localization criterion}. 
\begin{remark}
This example of  five
fusion rings indicates that one could exclude more than one fusion rings by checking the positivity of only one matrix.
This could not be done by the Schur product criterion. So localized criteria could improve the efficiency to check fusion rings.
\end{remark}

\section{Reduced Twisted Criteria}\label{sec:reduced twist criteria}
In this section, we twist the fusion (sub)matrices and take Hadamard product of fusion (sub)matrices to obtain much more powerful criteria. Recall that the computational complexity of checking fusion rings of rank $m$ by primary $n$-criterion
 grows exponentially $O(m^{3n})$. Note that the order of the matrix remains the same under the Hadamard product. 
 So we apply the Hadamard product to primary criteria and reduce the computation complexity
to $O(m^3+m^2 \log n)$.  Moreover, the twisting of the fusion (sub)matrices leads to corresponding criteria, which allow us to
 exclude fusion rings with unknown data.
 \subsection{Results and Examples}
\begin{theorem}[Reduced twisted criteria]\label{thm:hadamard criterion}
Suppose $\mathcal{R}$ is a fusion ring with basis $\{x_1=1,x_2,\ldots,x_m\}$  and $M_i$ is the fusion matrix of $x_i$. Let $S\subseteq \{1,2,\ldots,m\}$ be a subset and $U_j\in M_{|S|}(\mathbb{C})$  are unitary matrices.
  If $\mathcal{R}$ admits a unitary categorification, then for any $n \geq 1$, we have
\begin{align}\label{eq:n positivity locol criterion hadamard}
\sum_{i=1}^m \|M_i\|^2\bigg(\frac{ M_i^S }{\|M_i\|}\bigg)^{\ast n} \geq & 0, \quad \text{ Reduced }\\
\sum_{i=1}^m \|M_i\|^2 U_j^{\otimes_{j=1}^n} \bigg(\frac{M_i^S}{\|M_i\|}\bigg)^{\otimes n} (U_j^*)^{\otimes_{j=1}^n}  \geq & 0,  \text{ Twisted } \\
\sum_{i=1}^m \|M_i\|^2\bigg(\frac{U_j M_i^S U_j^*}{\|M_i\|}\bigg)^{\ast_{j=1}^n} \geq & 0,  \text{ Reduced twisted }
\end{align}
where $\ast$ is the Hadamard product of matrices.
\end{theorem}
\begin{proof}
Define
\[S^{(n)}:=\{(\underbrace{i,i,\dots,i}_n):\ i\in S\},\quad S^{(1)}=S.\]
Then
\[(U_j M_i^S U_j^*)^{\ast_{j=1}^n}=P_{S^{(n)}}(U_j M_i^S U_j^*)^{\otimes_{j=1}^n} P_{S^{(n)}}.\]
Theorem \ref{thm:localization criterion} implies
\begin{align*}
\sum_{i=1}^m \|M_i\|^2\bigg(\frac{U_j M_i^S U_j^*}{\|M_i\|}\bigg)^{\ast_{j=1}^n}=
P_{S^{(n)}}  U_j^{\otimes_{j=1}^n}\sum_{i=1}^m \|M_i\|^2  \bigg(\frac{ M_i^S}{\|M_i\|}\bigg)^{\otimes n}  U_j^{*\otimes_{j=1}^n} P_{S^{(n)}}\bigg|_{S^{(n)}\times S^{(n)}}\geq0 .
\end{align*}
This completes the proof of the theorem.
\end{proof}

\begin{remark}
If there are enough zero entries of the fusion matrices, then we can eliminate the unknown  entries of the fusion matrices
by choosing proper unitary matrices.
\end{remark}

\begin{remark}
In Theorem \ref{thm:hadamard criterion}, when the fusion matrices $M_j$ is twisted by unitary matrices $U_j$, we have
\begin{align*}
\sum_{i=1}^m \|M_i\|^2  \bigg(\frac{ (U_jM_i U_j^*)^S}{\|M_i\|}\bigg)^{\otimes_{j=1}^n}   \geq 0.
\end{align*}
This gives a different twisted positivity for a fusion ring.
\end{remark}

Now we apply the criteria to a family of fusion rings.
\begin{theorem}\label{thm:two special simple object fusion ring}
Let $\mathcal{R}$ be a fusion ring with basis $\{x_1=1,x_2,\ldots,x_m\}$  and $M_i$ is the fusion matrix of $x_i$. 
Suppose that  
\begin{align*}
x_2 x_2^*=& I+sx_2+\ell x_3\\
x_3 x_3^*=& I+tx_2+k x_3+\sum_{i=4}^m \lambda_i x_i
\end{align*}
for $s, t, k, \lambda_i \geq 0,\ell \geq 1$, $i\geq 4$.
Then we have the following statements:
\begin{enumerate}
\item If $x_2$ is self-dual and and there exist $n\geq2$, $a,b\geq1$, $a+b=n$, such that
\begin{align}\label{eq:AB local}
\bigg(\frac{s^a t^b}{d_2^{n-2}}+\frac{\ell^a k^b}{d_3^{n-2}}+1\bigg)\bigg(\frac{s^b t^a}{d_2^{n-2}}+\frac{\ell^b k^a}{d_3^{n-2}}+1\bigg)-\bigg(\frac{\ell^n}{d_2^{n-2}}+\frac{t^n}{d_3^{n-2}}\bigg)^2<0,
\end{align}
 then $\mathcal{R}$ admits no unitary categorification.
\item If $x_2$ is not self-dual and there exist $n\geq2$, $a,b\geq1$, $a+b=n$, such that
 \begin{align}\label{eq:AA* local}
 \frac{2s^n}{d_2^{n-2}}+1-\frac{s^a\tau(x_2^3)^b+s^b\tau(x_2^3)^a}{d_2^{n-2}}<0,
 \end{align}
  then $\mathcal{R}$ admits no unitary categorification, where $\tau$ is the faithful tracial state on $\mathcal{R}$.
\end{enumerate}
\end{theorem}
\begin{proof}
(1)
If $x_2$ is self-dual, then $x_3$ is also self-dual.
Take the local data $S=\{2,3\}$. 
Then
\begin{align*}
M_1^S=\begin{pmatrix}
1&0\\0&1
\end{pmatrix},M_2^S=\begin{pmatrix}
s&\ell\\ \ell&t
\end{pmatrix},
M_3^S=\begin{pmatrix}
\ell&t
\\t&k
\end{pmatrix},
M_{i}^S=\begin{pmatrix}
0&\ast \\
\ast&\lambda_i
\end{pmatrix},\quad i\geq4.
\end{align*}
Let $U_1=\cdots=U_a=I$, $U_{a+1}=\cdots=U_n=\begin{pmatrix}
0&1 \\1 &0
\end{pmatrix}$.
By the assumption that $a, b \geq 1$ and  $\{x_i: 4\leq i\leq m\}=\{x_i^*: 4\leq i\leq m\}$, we have 
\begin{align*}
\sum_{i=4}^m \frac{1}{d_i^{n-2}} (U_j M_{i}^S U_j^*)^{\ast_{j=1}^n}=
\begin{pmatrix}
0&\lambda \\\lambda &0
\end{pmatrix}
\end{align*}
for some $\lambda >0$.
Applying Theorem \ref{thm:hadamard criterion} to $\mathcal{R}$, we have
\begin{align*}
&\det\bigg(\sum_{i=1}^m \frac{1}{d_i^{n-2}} (U_j M_{i}^S U_j^*)^{\ast_{j=1}^n}\bigg)\\
\leq &\det\bigg[ \begin{pmatrix}
1&0\\0&1
\end{pmatrix}
+\frac{1}{d_2^{n-2}}\begin{pmatrix}
s^a t^b&\ell^n\\ \ell^n&s^b t^a
\end{pmatrix}+
\frac{1}{d_3^{n-2}}\begin{pmatrix}
\ell^a k^b&t^n\\t^n&\ell^b k^a
\end{pmatrix} + \begin{pmatrix}
0&\lambda \\\lambda &0
\end{pmatrix}\bigg] \\
\leq & \bigg(\frac{s^a t^b}{d_2^{n-2}}+\frac{\ell^a k^b}{d_3^{n-2}}+1\bigg)\bigg(\frac{s^b t^a}{d_2^{n-2}}+\frac{\ell^b k^a}{d_3^{n-2}}+1\bigg)-\bigg(\frac{\ell^n}{d_2^{n-2}}+\frac{t^n}{d_3^{n-2}}\bigg)^2 \\
<&0.\quad\quad\quad \quad\quad\text{(From inequality \eqref{eq:AB local})}
\end{align*}
Hence $\mathcal{R}$ admits no unitary categorification.

(2) If $x_2$ is not self-dual, then $x_3=x_2^*$ and $s=k=t=\ell$. Take the local data $S=\{2, 3\}$. 
Then
\begin{align*}
M_1^S=\begin{pmatrix}
1&0\\0&1
\end{pmatrix},M_2^S=\begin{pmatrix}
s&\tau(x_2^3)\\ s&s
\end{pmatrix},M_{3}^S=\begin{pmatrix}
s&s\\ \tau(x_2^3)&s
\end{pmatrix},M_{i}^S=\begin{pmatrix}
0&\ast\\\ast&\lambda_i
\end{pmatrix},\quad i\geq4.
\end{align*}
With the same process, we could obtain the conclusion.
\end{proof}

\begin{remark}
In the proof of Theorem \ref{thm:two special simple object fusion ring}, the unknown data $\lambda_i$ is removed by the following equation:
\begin{align*}
 (U_j M_{i}^S U_j^*)^{\ast_{j=1}^n}=
\begin{pmatrix}
0&* \\
*&0
\end{pmatrix}.
\end{align*}
So the reduced twist criteria allow us to check fusion rings with unknown data, which is superior to the Schur product criterion.
 \end{remark}

In the following, we give an example of a family of fusion rings excluded by reduced twisted criteria.
\begin{theorem}\label{thm:R4k}
Suppose $\mathcal{R}_{4,k}$ is a ring generated by the following $4$ matrices:
\begin{align*}
\begin{pmatrix}
1&0&0&0\\
0&1&0&0\\
0&0&1&0\\
0&0&0&1
\end{pmatrix},
\begin{pmatrix}
0&1&0&0\\
1&k&0&1\\
0&0&k&1\\
0&1&1&k
\end{pmatrix},
\begin{pmatrix}
0&0&1&0\\
0&0&k&1\\
1&k&0&0\\
0&1&0&k
\end{pmatrix},
\begin{pmatrix}
0&0&0&1\\
0&1&1&k\\
0&1&0&k\\
1&k&k&1
\end{pmatrix}.
\end{align*}
Then $\mathcal{R}_{4,k}$ is a simple fusion ring admitting no unitary categorification for any $k\geq3$.
\end{theorem}

\begin{proof}
By computations, we have the following fusion rules:
\begin{align*}
x_2^2&=1+kx_2+x_4,\\
x_3^2&=1+kx_2,\\
x_4^2&=1+k(x_2+x_3)+x_4,\\
x_2x_3&=kx_3+x_4,\\
x_3x_4&=x_2+kx_4,\\
x_2x_4&=x_2+x_3+kx_4.
\end{align*}
Take local data S=$\{2,3\}$. Then $\mathcal{R}_{4,k}$
satisfies the condition in Theorem \ref{thm:two special simple object fusion ring} by switching $x_2$ and $x_3$.
By the Collatz-Wielandt formula for non-negative matrix, we have $d_3=\|x_3\|\leq k+1\leq\|x_2\|=d_2$. 
Then
\begin{align*}
\sqrt{k^2+k+1}\leq \|x_3\|\leq k+1.
\end{align*}
Inequality \eqref{eq:AB local} ($n=3$) is equivalent to
\begin{align*}
f(\|x_3\|)=\|x_3\|^3-k^3\|x_3\|^2+(k^4-1)\|x_3\|+k^3<0.
\end{align*}
Note that $f'(x)<0$ for any $x\in[\sqrt{k^2+k+1},k+1]$ when $k\geq3$.
It is not difficult to check that $f(\sqrt{k^2+k+1})<0$ when $k\geq 5$. So $f(x)<0$
for any $x\in[\sqrt{k^2+k+1},k+1]$ when $k\geq5$.
Thus,
 $f(\|x_3\|)<0$ and $\mathcal{R}_{4,k}$ admits no unitary categorification when $k\geq 5$.
On the other hand, $\mathcal{R}_{4,3}$ and $\mathcal{R}_{4,4}$ could be excluded by direct computations.
\end{proof}
\begin{remark}
The reduced twist criteria make it applicable to exclude a family of infinitely many simple fusion rings. This is not easy 
to complete by the Schur product criterion since 
 the  character table of $\mathcal{R}_{4,k}$ are very complex and it is hard to compute.
\end{remark}
\subsection{Inequivalence of Localized Criteria}\label{sec:locineq}
Let $\mathcal{R}$ be any one of the five fusion rings of rank 8 in subsection \ref{sub:S8}.
Recall that $S=\{x_7,x_8\}$.
Applying Theorem \ref{thm:hadamard criterion}
for $n=3$, $U_1=U_2=I$, $U_3=\begin{pmatrix}
0&1\\1&0
\end{pmatrix}$, we obtain
\begin{align*}
\sum_{i=1}^8\frac{1}{\|M_i\|} (U_j M_i^S U_j^*)^{\ast_{j=1}^3}=\frac{1}{330}\begin{pmatrix}
4286&4101\\
4101&3736
\end{pmatrix}.
\end{align*}
It is not positive since its determinant $\simeq -7.4$. 
Hence we see that these five fusion rings admit no unitary categorifications by reduced twisted criteria.
For $n=4$, $U_1=U_2=I$, $U_3=U_4=\begin{pmatrix}
0&1\\1&0
\end{pmatrix}$, we have that
\begin{align*}
\sum_{i=1}^8\frac{1}{\|M_i\|^2} (U_j M_i^S U_j^*)^{\ast_{j=1}^4}\simeq\begin{pmatrix}
4.85&4.88\\
4.88&4.85
\end{pmatrix}.
\end{align*}
It is not positive either. 

Now let us consider the limit of $\displaystyle \sum_{i=1}^n\frac{1}{\|M_i\|^2} (U_j M_i^S U_j^*)^{\ast_{j=1}^n}$ when $n$ goes to infinity.
\begin{proposition}
Suppose $\mathcal{R}$ is a fusion ring with basis $\{x_1=I,x_2,\ldots,x_m\}$  and $M_i$ is the fusion matrix of $x_i$. Let $S\subseteq \{1,2,\ldots,m\}$ be a subset. 
Define $L_i^S=(a_{kj})_{k,j=1}^m \in M_m(\mathbb{C})$ such that
\begin{align*}
a_{kj}=\begin{cases}
1&\text{if $j$ or $k\in S$ and $N_{i,j}^k=\|M_i\|$},\\
0&\text{otherwise}.
\end{cases}
\end{align*}
If $\mathcal{R}$ admits a unitary categorification, then
\begin{align}\label{eq:limit criterion}
\sum_{i=1}^m \|M_i\|^2 L_i^S\geq0.
\end{align}
\end{proposition}
\begin{proof}
Note that $N_{i,j}^k\leq\|M_i\|$. Let $U_i=I$.
Theorem \ref{thm:hadamard criterion} implies that
\begin{align*}
\sum_{i=1}^m \|M_i\|^2 L_i^S=\lim_{n\to \infty}\sum_{i=1}^m \|M_i\|^2  \bigg(\frac{ M_i^S}{\|M_i\|}\bigg)^{\ast n} \geq 0.
\end{align*}
This completes the proof of the proposition.
\end{proof}

Note that for any unitary matrices $U_j$, we have
\begin{align*}
\lim_{n\to \infty}\sum_{i=1}^8\frac{1}{\|M_i\|^{n-2}} (U_j M_i^S U_j^*)^{\ast_{j=1}^n}=I.
\end{align*}
We see that it is positive when $n$ is sufficiently large. 
This implies localized criteria are not all equivalent.
\section{Multifusion rings and principal graphs}\label{sec:multifusion ring and principal graph}
The Grothendieck ring of a multifusion category is a multifusion ring.
We could ask whether a multifusion ring is a Grothendieck ring of a multifusion category.
By similar arguments, we see that the criteria obtained in this paper are true for unitary multifusion categorification of a multifusion ring.
The principal graph of a subfactor provides partial data of the multifusion ring of its bimodule category, see \S \ref{Sec: Principal Graph} in the Appendix. We apply the localized criteria to provide obstruction of bipartite graphs from being principle graphs for subfactors of finite depth.

\subsection{Results and Examples}
By a similar argument in the proof of Theorem \ref{thm:Higher Positivity Criterion}, we have the following result.
\begin{theorem}\label{thm:primary criteria multifusion ring}
Suppose $\mathcal{R}$ is a multifusion ring with basis $\{x_1, x_2,\ldots,x_m\}$  and $M_j$ is the fusion matrix of $x_j$. 
If $\mathcal{R}$ admits a unitary multifusion categorification, then for any $n \geq 1$, we have
\begin{align}\
\sum_{j=1}^m \|M_j\|^2\bigg(\frac{M_j}{\|M_j\|}\bigg)^{\otimes n}\geq 0.
\end{align}
\end{theorem}

All localized criteria for unitary categorification of fusion rings obtained before in this paper are true for unitary multifusion
 categorification of a multifusion ring.
Hence the localized positivity are obstructions for constructing principal graphs of subfactors of finite depth.
Applying the idea of localization from Theorem \ref{thm:localization criterion}, we have
\begin{theorem}\label{thm:local criteria multifusion ring}
Suppose $\mathcal{R}$ is a multifusion ring with basis $\{x_1, x_2,\ldots,x_m\}$  and $M_j$ is the fusion matrix of $x_j$. 
Let $S\subseteq\{1,2,\cdots,m\}$ be a subset. 
If $\mathcal{R}$ admits a unitary multifusion categorification, then for any $n \geq 1$, we have
\begin{align}\label{eq:multi local}
\sum_{j} d_j^2\bigg(\frac{M_j^S}{d_j}\bigg)^{\otimes n}\geq 0,
\end{align}
where $d_j$ is the quantum dimension of $x_j$ for $j=1, \ldots .$
\end{theorem}
\begin{example}
Let $D_5$ be the following bipartite graph: $q=e^{\pi i/8}$,
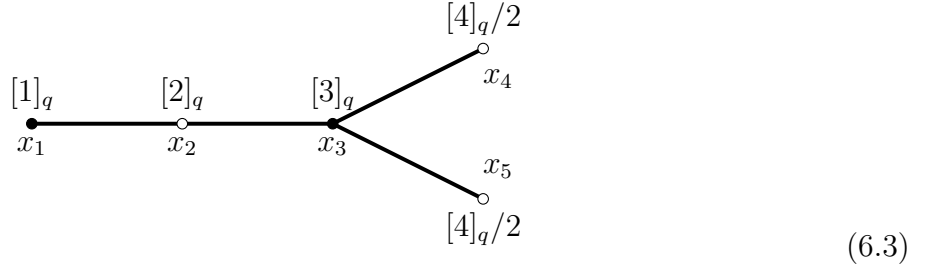
\begin{align}\label{fig:principle graph 1}
\begin{tikzpicture}
\draw(0,0)node{\tiny$\bullet$};
\draw(2,0)node{\tiny$\bullet$};
\draw(4,0)node{\tiny$\bullet$};
\draw(0,-0.3)node{$x_1$};
\draw(2,-0.3)node{$x_2$};
\draw(4,-0.3)node{$x_3$};
\draw[line width=1.5](0,0) --(2,0);
\draw[line width=1.5](2,0) --(4,0);
\draw[line width=1.5](4,0)--(6,1);
\draw[line width=1.5](4,0)--(6,-1);
\draw(6.2,0.6)node{$x_4$};
\draw(6.2,-0.6)node{$x_5$};
\draw(0,0.4)node{$[1]_q$};
\draw(2,0.4)node{$[2]_q$};
\draw(4,0.4)node{$[3]_q$};
\draw(6,1.4)node{$[4]_q/2$};
\draw(6,-1.4)node{$[4]_q/2$};
\begin{scope}
\draw [fill=black] (0, 0) circle (0.07);
\end{scope}
\begin{scope}[shift={(2, 0)}]
\draw [fill=white] (0, 0) circle (0.07);
\end{scope}
\begin{scope}[shift={(4, 0)}]
\draw [fill=black] (0, 0) circle (0.07);
\end{scope}
\begin{scope}[shift={(6, 1)}]
\draw [fill=white] (0, 0) circle (0.07);
\end{scope}
\begin{scope}[shift={(6, -1)}]
\draw [fill=white] (0, 0) circle (0.07);
\end{scope}
\end{tikzpicture}
\end{align}
We take the local set $S=\{x_1,x_4\}$. 
Then
\begin{align*}
T_1^S=\sum_{j=1}^5  d_j M_j^S=\begin{pmatrix}
1&[4]_q/2\\ [4]_q/2&1
\end{pmatrix}
\end{align*}
is not positive. So $D_5$ is not a principal graph of a subfactor.
\end{example}
Applying the idea of Theorem \ref{thm:hadamard criterion}, we further have
\begin{theorem}\label{thm:multifusion hadamard criterion}
Suppose $\mathcal{R}$ is a multifusion ring with basis $\{x_1, x_2,\ldots,x_m\}$  and $M_j$ is the fusion matrix of $x_j$. 
Let $S\subseteq\{1,2,\cdots,m\}$ be a subset, and $U_j\in M_{|S|}(\mathbb{C})$  are unitary matrices.
  If $\mathcal{R}$ admits a unitary multifusion categorification, then for any $n \geq 1$, we have
\begin{align}\label{eq:multi hadamard}
\sum_{i} d_i^2\bigg(\frac{ M_i^S }{d_i}\bigg)^{\ast n} \geq & 0, \quad \text{ Reduced }\\
\sum_{i} d_i^2 U_j^{\otimes_{j=1}^n} \bigg(\frac{M_i^S}{d_i}\bigg)^{\otimes n} (U_j^*)^{\otimes_{j=1}^n}  \geq & 0,  \text{ Twisted } \\
\sum_{i} d_i^2\bigg(\frac{U_j M_i^S U_j^*}{d_i}\bigg)^{\ast_{j=1}^n} \geq & 0,  \text{ Reduced twisted }
\end{align}
where $\ast$ is the Hadamard product of matrices.
\end{theorem}
Note that multifusion categories are multicolored.
For graded multifusion rings, we have
\begin{theorem}
Suppose $\mathcal{R}$ is a $\mathbb{Z}_k$-graded multifusion ring with basis $\{x_1, x_2,\ldots,x_m\}$  and $M_j$ is the fusion matrix of $x_j$. 
Let $S\subseteq\{1,2,\cdots,m\}$ be a subset.
If $\mathcal{R}$ admits a  unitary multifusion categorification, then
\begin{align}\label{eq:multi graded}
\sum_{i=1}^m (-1)^{|i|}d_i^2\bigg(\frac{M_i^S}{d_i}\bigg)^{\otimes n}\geq 0,
\end{align}
where $|i|$ is the grading of $x_i$ and $d_i$ is the quantum dimension of $x_i$ for $i=1, \ldots .$.
\end{theorem}

\subsection{Bipartite Graphs with Local data Excluded}
In this subsection, we will present a family of bipartite graphs excluded from principal graphs of subfactors by the reduced twisted criteria.
\begin{theorem}\label{thm:infinite principle graph exclude}
Let $\mathcal{G}$ be a bipartite graph of finite depth with the following local form $\ell\geq1$:
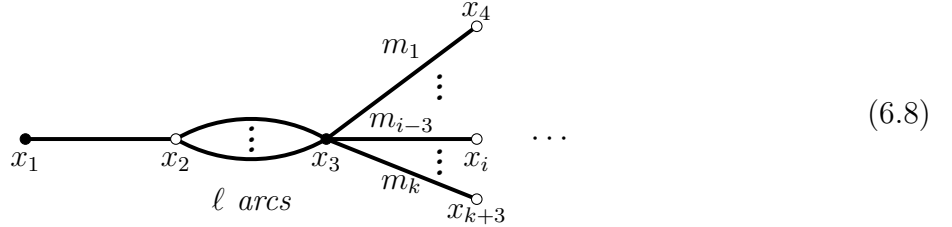
\begin{align}\label{fig:principle graph 1}
\raisebox{-1.5cm}{
\begin{tikzpicture}
\draw[line width=1.5](0,0) --(2,0);
\draw(0,-0.3)node{$x_1$};
\draw(2,-0.3)node{$x_2$};
\draw(4,-0.3)node{$x_3$};
\draw (2,0)[line width=1.5] arc (120:60:2);
\draw (2,0)[line width=1.5] arc (-120:-60:2);
\draw(3,0.1)node{\textbf{$\vdots$}};
\draw(3,-0.8)node{$\ell$ arcs};
\draw[line width=1.5](4,0) --(6,0);
\draw(6,-0.3)node{$x_{i}$};
\draw[line width=1.5](4,0) --(6,-0.8);
\draw(6,-1)node{$x_{k+3}$};
\draw(5.5,-0.2)node{\textbf{$\vdots$}};
\draw[line width=1.5](4,0)--(6,1.5);
\draw(6,1.7)node{$x_4$};
\draw(5.5,0.8)node{\textbf{$\vdots$}};
\begin{scope}
\draw [fill=black] (0, 0) circle (0.07);
\end{scope}
\begin{scope}[shift={(2, 0)}]
\draw [fill=white] (0, 0) circle (0.07);
\end{scope}
\begin{scope}[shift={(4, 0)}]
\draw [fill=black] (0, 0) circle (0.07);
\end{scope}
\begin{scope}[shift={(6,-0.8)}]
\draw [fill=white] (0, 0) circle (0.07);
\end{scope}
\begin{scope}[shift={(6,0)}]
\draw [fill=white] (0, 0) circle (0.07);
\end{scope}
\begin{scope}[shift={(6,1.5)}]
\draw [fill=white] (0, 0) circle (0.07);
\end{scope}
\draw(7,0)node{\textbf{$\cdots$}};
\draw(5,1.2)node{$m_1$};
\draw(5,0.2)node{$m_{i-3}$};
\draw(5,-0.6)node{$m_{k}$};
\end{tikzpicture}}
\end{align}
Let $M=\sum_{j=1}^k m_i^2+\ell^2-1$ and $d_2$ is the quantum dimension of $x_2$. 
If there exist $n\geq2$, $a,b\geq1$, $a+b=n$, such that
\begin{align}\label{eq:ab}
\bigg(\frac{\ell^{2a-2}M^b}{(d_2^2-1)^{n-2}}+1  \bigg)
\bigg(\frac{\ell^{2b-2}M^a}{(d_2^2-1)^{n-2}}+1  \bigg)-\frac{\ell^{2n}}{d_2^{2n-4}}<0,
\end{align}
then $\mathcal{G}$ can not be a principal graph of a subfactor.
\end{theorem}
\begin{proof}
Suppose that $\mathcal{G}$ is a principal graph of a subfactor. 
Let $\mathcal{R}$ be the multifusion ring arising from $\mathcal{G}$, generated by $\{x_1,x_2,\cdots\}$ with  fusion matrices $M_i$,
where $x_i$ is the direct sum of two irreducible bimodules. Let $d_i$ be the quantum dimension of $x_i$.
 Note that $x_2,x_3$ are self-dual. We have
\begin{align*}
  x_2^2&=1+\ell x_3\\
x_3^2&=1+\dfrac{M}{\ell}x_3+\sum_{i\geq k+4}\lambda_i x_i.
\end{align*}
Then $d_3=(d_2^2-1)/\ell$.
Take local data $S=\{2,3\}$, we have
\begin{align*}
M_1^S=\begin{pmatrix}
1&0\\0&1
\end{pmatrix},M_2^S=\begin{pmatrix}
0&\ell\\ \ell&0
\end{pmatrix},
M_3^S=\begin{pmatrix}
\ell&0
\\0&\dfrac{M}{\ell}
\end{pmatrix},
M_{i}^S=\begin{pmatrix}
0&\ast \\
\ast&\lambda_i
\end{pmatrix},\quad i\geq4.
\end{align*}
Let $U_1=\cdots=U_a=\begin{pmatrix}
1&0\\0&1
\end{pmatrix}$ and 
$U_{a+1}=\cdots=U_n=\begin{pmatrix}
0&1\\1&0
\end{pmatrix}$.
Since $\{x_i:i\geq4\}=\{x_i^*:i\geq4\}$, we have
\begin{align*}
\sum_{i\geq 4}\dfrac{1}{d_i^{n-2}}(U_j M_i^S U_j^*)^{\ast_{j=1}^n}=\begin{pmatrix}
0&\lambda\\
\lambda&0
\end{pmatrix}
\end{align*}
for some $\lambda>0$. So from Inequality \eqref{eq:ab},
\begin{align*}
&\det\sum_{i=1}^m \frac{1}{d_i^{n-2}} (U_j M_{i}^S U_j^*)^{\ast_{j=1}^n}
\leq \det\sum_{i=1}^3 \frac{1}{d_i^{n-2}} (U_j M_{i}^S U_j^*)^{\ast_{j=1}^n}<0.
\end{align*}
By Theorem \ref{thm:multifusion hadamard criterion}, $\mathcal{R}$ admits no unitary multifusion categorification. This leads to a contradiction.
\end{proof}

\begin{remark}
The key point in the proof of Theorem \ref{thm:infinite principle graph exclude} is that we use the reduced twist criterion to remove the unknown data $\lambda_i$. So it is practicable to check the positivity of the matrix even some entries are known. This is a big advantage of  the reduced twist criterion.
\end{remark}

\begin{remark}\label{rem:infinite principle graph exclude}
Suppose the bipartite graph $\mathcal{G}$  in Theorem \ref{thm:infinite principle graph exclude}
is a  principal graph of a subfactor with finite index $\delta^2$,  we have $M=\dim_{\mathbb{C}}(\mathscr{P}_{3,\pm})-
\ell^4-\ell^2-2$ and $\dim_{\mathbb{C}}(\mathscr{P}_{2,\pm})=\ell^2+1$. So Theorem \ref{thm:infinite principle graph exclude} indicates that there does not exist a principal graph of a subfactor of index $\delta^2$ such that
\begin{enumerate}
\item  It has the following local graph
\begin{align}
\raisebox{-1cm}{
\begin{tikzpicture}
\draw[line width=1.5](0,0) --(2,0);
\draw(0,-0.3)node{$x_1$};
\draw(2,-0.3)node{$x_2$};
\draw(4,-0.3)node{$x_3$};
\draw (2,0)[line width=1.5] arc (120:60:2);
\draw (2,0)[line width=1.5] arc (-120:-60:2);
\draw(3,0.1)node{\textbf{$\vdots$}};
\draw(3,-0.8)node{$\ell$ arcs};
\begin{scope}
\draw [fill=black] (0, 0) circle (0.07);
\end{scope}
\begin{scope}[shift={(2, 0)}]
\draw [fill=white] (0, 0) circle (0.07);
\end{scope}
\begin{scope}[shift={(4, 0)}]
\draw [fill=black] (0, 0) circle (0.07);
\end{scope}
\draw(5,0)node{\textbf{$\cdots$}};
\end{tikzpicture}}
\end{align}
\item  There exist $n\geq2$, $a,b\geq1$, $a+b=n$, such that
\begin{align}\label{eq:principal exlcude}
\bigg(\frac{\ell^{2a-2}M^b}{(\delta^2-1)^{n-2}}+1  \bigg)
\bigg(\frac{\ell^{2b-2}M^a}{(\delta^2-1)^{n-2}}+1  \bigg)-\frac{\ell^{2n}}{\delta^{2n-4}}<0.
\end{align}
\end{enumerate}
In particular, let $a=b=2$. Suppose that $\ell^2\geq\delta$.
Then Inequality \eqref{eq:principal exlcude} is equivalent to
\begin{align}
\dim_{\mathbb{C}}{\mathscr{P}_{3,\pm}}<\dfrac{\delta^2-1}{\ell}\sqrt{\dfrac{\ell^4}{\delta^2}-1}+\ell^4+\ell^2+2.
\end{align}
So Theorem \ref{thm:infinite principle graph exclude} allows us to
eliminate a large family of bipartite graphs from principal
graphs of subfactors with certain dimension bounds.
\end{remark}

Next, we present a special case of Theorem \ref{thm:infinite principle graph exclude} when $a=b=2$ and $n=4$.
\begin{corollary}\label{cor:infinite principle graph exclude}
Suppose $\mathcal{G}$ is the bipartite graph defined in Theorem \ref{thm:infinite principle graph exclude}.  If 
\begin{align}\label{eq:dml}
\frac{\sqrt{2}M}{\ell}<d_2-\dfrac{1}{d_2}<d_2<\dfrac{\ell^2}{\sqrt{2}},
\end{align}
then 
$\mathcal{G}$ can not be a principal graph of a subfactor.
\begin{proof}
When $a=b=2$, inequality \eqref{eq:ab} is equivalent to 
\begin{align*}
\dfrac{\ell^{2}M^2}{(d_2^2-1)^{2}}+1<\dfrac{\ell^4}{d_2^{2}}.
\end{align*}
It is not difficult to check that inequality \eqref{eq:dml} indicates
\begin{align*}
\dfrac{\ell^{2}M^2}{(d_2^2-1)^{2}}<\dfrac{1}{2}\cdot\dfrac{\ell^4}{d_2^{2}},\quad \text{and}\quad 1<\dfrac{1}{2}\cdot\dfrac{\ell^4}{d_2^{2}}.
\end{align*}
Adding this two inequalities, we know that inequality \eqref{eq:ab} holds. So from Theorem \ref{thm:infinite principle graph exclude},
$\mathcal{G}$ can not be a principal graph of a subfactor.
\end{proof}
\end{corollary}

\begin{remark}\label{rem:n+1 stronger n}
Suppose that $M\leq A\ell^2$, where $A$ is a positive constant. If
\begin{align*}
\sqrt{2}A\ell<d_2<\dfrac{\ell^2}{\sqrt{2}},
\end{align*}
then inequality \eqref{eq:dml} holds. So $\mathcal{G}$ can not be a principal graph of a subfactor.
\end{remark}

\section{Questions}\label{sec:further}

Unitary multifusion categories are unitary multitensor categories \cite{Pen20} whose ranks could be infinite and the Grothendieck ring of a unitary multitensor categories is a based ring.
A natural question is whether a based ring is a Grothendieck ring of a unitary multitensor category.
Inspired by localized criteria for multifusion rings in Section \ref{sec:multifusion ring and principal graph}, we propose the following question.
\begin{question}
Suppose $\mathcal{R}$ is a based ring with basis $\{x_1, x_2,\ldots\}$.
Let $S$ be a finite subset of $\mathbb{N}$ and $M_j^S$ is the fusion matrix of $x_j$ restricted on $S$. 
If $\mathcal{R}$ admits a unitary multitensor categorification, does the following inequality hold for all $n \in \mathbb{N}$:
\begin{align}\label{eq:multi local}
\sum_{j} d_j^2\bigg(\frac{M_j^S}{d_j}\bigg)^{\otimes n}\geq 0?
\end{align}
\end{question}

Note that these unitary categorification criteria are also criteria of subfactorization of fusion bialgebras (See Definition 7.3 in \cite{LPW21}).
There are subfactorizable fusion rings that admit no unitary categorification. For example, the subfactor  planar algebra $\mathbb{Z}_{n}\ast TL(\delta)$, such that $\delta^2-1=\sqrt{\delta^2-1}+n$, is a subfactorization of the following near group fusion ring, 
$$\mathcal{R}_n:=\{b, g^k, k\in \mathbb{Z}_{n}: g^{n}=1, gb=bg=b, b^2=b+\sum_{k=1}^{n} g^k\}.$$ 
We see that $\mathcal{R}_n$ passes primary criteria because it is subfactorizable.  
However, $\mathcal{R}_n$ admits no unitary categorification when $n \geq 2$, due to Theorem 1.1 in \cite{Izu15}.

Bisch-Haagerup constructed the following sequence of fusion rings in 1994,
\begin{align*}
BH_r:=\left\langle \rho,\tau:\ \rho^2=\rho+1,\tau^2=1,(\rho\tau)^r=(\tau \rho)^r\right\rangle.
\end{align*}
and they asked that whether they have unitary categorifications. 
It had been widely believed that the whole sequence of fusion rings have unitary categorification. A surprising answer was given in \cite{Liu15} that $BH_r$ admits a unitary categorification if and only if $r=0,1,2,3$. (This is a slightly weaker version of the original question and answer on possible compositions of $A_3$ and $A_4$ subfactors.) It is natural to ask the following question.

\begin{question} \label{Q4}
For each $r\geq 4$, does $BH_r$ pass the primary $3$-criterion? If so, then is $BH_r$ subfactorizable?
\end{question}

\section{Appendix}\label{sec:Appendix}
\subsection{SageMath computations}
\footnotesize
\subsubsection{Computation for Proposition \ref{prop:nc6}} \label{App:SNC6}
\begin{verbatim}
sage: MM=[
....: [[1,0,0,0,0,0],[0,1,0,0,0,0],[0,0,1,0,0,0],[0,0,0,1,0,0],[0,0,0,0,1,0],[0,0,0,0,0,1]],
....: [[0,1,0,0,0,0],[1,4,2,2,2,2],[0,2,2,1,2,4],[0,2,1,2,4,2],[0,2,2,4,5,4],[0,2,4,2,4,5]],
....: [[0,0,1,0,0,0],[0,2,2,1,4,2],[0,1,3,1,3,3],[1,2,3,3,1,3],[0,2,3,3,5,4],[0,4,1,3,4,5]],
....: [[0,0,0,1,0,0],[0,2,1,2,2,4],[1,2,3,3,3,1],[0,1,1,3,3,3],[0,4,3,1,5,4],[0,2,3,3,4,5]],
....: [[0,0,0,0,1,0],[0,2,4,2,5,4],[0,4,1,3,5,4],[0,2,3,3,5,4],[1,5,5,5,5,7],[0,4,4,4,7,7]],
....: [[0,0,0,0,0,1],[0,2,2,4,4,5],[0,2,3,3,4,5],[0,4,3,1,4,5],[0,4,4,4,7,7],[1,5,5,5,7,5]]
....: ]
sage: L=[matrix(m) for m in MM]
sage: dim=[m.norm() for m in L]
sage: M=sum((L[i].tensor_product(L[i])).tensor_product(L[i])/(dim[i]) for i in range(6))
sage: E=M.eigenvalues()
sage: E[24]
-1.176375045085803
\end{verbatim}

\subsubsection{Computation for Subsection \ref{sub:S7}}\label{App:S7}
\begin{verbatim}
sage: MM=[[[1,0],[0,1]],[[1,1],[1,1]],[[1,1],[1,1]],
....: [[1,1],[1,1]],[[2,1],[1,2]],[[0,3],[3,1]],[[3,1],[1,2]]]
sage: L=[matrix(m) for m in MM]
sage: dim=[1,5,5,5,6,7,7][1,1],[5,2],[10,2],[11,1],[12,2]
sage: M=sum((L[i].tensor_product(L[i])).tensor_product(L[i])/(dim[i]) for i in range(7))
sage: E=M.eigenvalues()
sage: E[4]
-0.6294949095094869
\end{verbatim}

\subsubsection{Computation for Subsection \ref{sub:S8}}\label{App:S8}
\begin{verbatim}
sage: MM=[[[1,0],[0,1]],[[1,1],[1,1]],[[1,1],[1,1]],[[2,2],[2,2]],
....: [[2,2],[2,2]],[[3,2],[2,3]],[[0,5],[5,1]],[[5,1],[1,4]]]
sage: L=[matrix(m) for m in MM]
sage: dim=[1,5,5,10,10,11,12,12]
sage: M=sum((L[i].tensor_product(L[i])).tensor_product(L[i])/(dim[i]) for i in range(8))
sage: E=M.eigenvalues()
sage: E[4]
-2.948812175750019
\end{verbatim}

\subsubsection{Computation for Lemma \ref{lem:1}} \label{sub:1}
\begin{verbatim}
sage: MM=[[[1,0],[0,1]],[[1,1],[1,1]],[[1,1],[1,1]],
....: [[1,1],[1,1]],[[2,1],[1,2]],[[0,3],[3,1]],[[3,1],[1,2]]]
sage: L=[matrix(m) for m in MM]
sage: var('x2,x3,x4,x5,x6,x7')
sage: dim=[1,x2,x3,x4,x5,x6,x7]
sage: M=sum((L[i].tensor_product(L[i])).tensor_product(L[i])/(dim[i]) for i in range(7))
sage: N=M[:7,:7]
sage: N.determinant().factor()
(x2*x3*x4*x5^3*x6^2*x7^3 + 67*x2*x3*x4*x5^3*x6^2*x7^2 + 18*x2*x3*x4*x5^3*x6*x7^3
+ 32*x2*x3*x4*x5^2*x6^2*x7^3 + 7*x2*x3*x5^3*x6^2*x7^3 + 7*x2*x4*x5^3*x6^2*x7^3
+ 7*x3*x4*x5^3*x6^2*x7^3 + 1025*x2*x3*x4*x5^3*x6^2*x7 + 216*x2*x3*x4*x5^3*x6*x7^2
+ 962*x2*x3*x4*x5^2*x6^2*x7^2 + 172*x2*x3*x5^3*x6^2*x7^2 + 172*x2*x4*x5^3*x6^2*x7^2
+ 172*x3*x4*x5^3*x6^2*x7^2 - 729*x2*x3*x4*x5^3*x7^3 - 126*x2*x3*x4*x5^2*x6*x7^3
- 90*x2*x3*x5^3*x6*x7^3 - 90*x2*x4*x5^3*x6*x7^3 - 90*x3*x4*x5^3*x6*x7^3
+ 195*x2*x3*x4*x5*x6^2*x7^3 + 54*x2*x3*x5^2*x6^2*x7^3 + 54*x2*x4*x5^2*x6^2*x7^3
+ 54*x3*x4*x5^2*x6^2*x7^3 + 3375*x2*x3*x4*x5^3*x6^2 - 7290*x2*x3*x4*x5^3*x6*x7
+ 4600*x2*x3*x4*x5^2*x6^2*x7 + 665*x2*x3*x5^3*x6^2*x7 + 665*x2*x4*x5^3*x6^2*x7
+ 665*x3*x4*x5^3*x6^2*x7 - 19683*x2*x3*x4*x5^3*x7^2 - 8478*x2*x3*x4*x5^2*x6*x7^2
- 2214*x2*x3*x5^3*x6*x7^2 - 2214*x2*x4*x5^3*x6*x7^2 - 2214*x3*x4*x5^3*x6*x7^2
+ 1833*x2*x3*x4*x5*x6^2*x7^2 + 432*x2*x3*x5^2*x6^2*x7^2 + 432*x2*x4*x5^2*x6^2*x7^2
+ 432*x3*x4*x5^2*x6^2*x7^2 - 5832*x2*x3*x4*x5^2*x7^3 - 729*x2*x3*x5^3*x7^3
- 729*x2*x4*x5^3*x7^3 - 729*x3*x4*x5^3*x7^3 - 1728*x2*x3*x4*x5*x6*x7^3
- 594*x2*x3*x5^2*x6*x7^3 - 594*x2*x4*x5^2*x6*x7^3 - 594*x3*x4*x5^2*x6*x7^3
+ 216*x2*x3*x4*x6^2*x7^3 + 63*x2*x3*x5*x6^2*x7^3 + 63*x2*x4*x5*x6^2*x7^3
+ 63*x3*x4*x5*x6^2*x7^3)*(x5^2*x6^2*x7^2 + 25*x5^2*x6^2*x7 - 9*x5^2*x6*x7^2
+ 12*x5*x6^2*x7^2 + 125*x5^2*x6^2 + 135*x5^2*x6*x7 + 120*x5*x6^2*x7 - 729*x5^2*x7^2
+ 108*x5*x6*x7^2 + 27*x6^2*x7^2)^2/(x2*x3*x4*x5^7*x6^6*x7^7)
sage: E=N.eigenvalues()
sage: E[4].factor()
1/2*(2*x5*x6*x7 + 25*x5*x6 - 9*x5*x7 + 12*x6*x7 - sqrt(125*x5^2*x6^2 + 9*(333*x5^2 - 72*x5*x6
+ 4*x6^2)*x7^2 - 30*(33*x5^2*x6 - 4*x5*x6^2)*x7))/(x5*x6*x7)
sage: S=solve((2*x5*x6*x7 + 25*x5*x6 - 9*x5*x7 + 12*x6*x7)^2 - 125*x5^2*x6^2 + 9*(333*x5^2
 - 72*x5*x6\ + 4*x6^2)*x7^2 - 30*(33*x5^2*x6 - 4*x5*x6^2)*x7,x6)
sage: SS=solve(x5^2*x6^2*x7^2 + 25*x5^2*x6^2*x7 - 9*x5^2*x6*x7^2 + 12*x5*x6^2*x7^2
+ 125*x5^2*x6^2\ + 135*x5^2*x6*x7 + 120*x5*x6^2*x7 - 729*x5^2*x7^2 + 108*x5*x6*x7^2
 + 27*x6^2*x7^2,x6)
sage: S==SS
True
\end{verbatim}
\subsubsection{Computation for Remark \ref{rk:1}} \label{sub:2}
\begin{verbatim}
sage: ddim=[1,5,5,5,6,x6,7]
sage: MM=sum((L[i].tensor_product(L[i])).tensor_product(L[i])/(ddim[i]) for i in range(7))
sage: MM.determinant().factor()
1/22136835840*(210556551*x6^4 - 966777308*x6^3 - 19784561832*x6^2 + 60951456720*x6
+ 306237561840)*(2495*x6^2 + 5544*x6 - 142884)^2/x6^8
sage: S=solve(210556551*x6^4 - 966777308*x6^3 - 19784561832*x6^2 + 60951456720*x6
 + 306237561840,x6)
sage: [s.rhs().n(digits=40) for s in S]
[-8.335767036491464773340488283566158298175 - 2.116481932246399247255558980887840928053e-42*I,
 -2.911937593466257812583289613470318862572 + 2.116481932246399247255558980887840928053e-42*I,
 6.245789619885724844645019429070095813233 - 3.429094729194210149876597305845817158163e-42*I,
 9.593447799759457999876934309097061722308 + 3.429094729194210149876597305845817158163e-42*I]
\end{verbatim}
The complex parts above are zero (they appear just due to numerical approximations).

\subsubsection{Computation for Remark \ref{rk:2}} \label{sub:3}
\begin{verbatim}
sage: R = ZZ['x6']
sage: x6 = R.gen()
sage: P=210556551*x6^4 - 966777308*x6^3 - 19784561832*x6^2 + 60951456720*x6 + 306237561840
sage: P.is_irreducible()
True
sage: a=421113102
....: b=44334061169015601
....: c=250462485504
....: d=345734333761887148583413
....: e=5213988190704773354123819324759655
....: f=64964979666121194605838007808
....: g=149108026701745210176
....: h=3010843992410706004
....: i=16567558522416134464
....: j=4926006796557289
....: k=1151178558485955841275891856
....: l=6021687984821412008
....: m=241694327
....: n=210556551
....: A=(c/d*I*sqrt(e) + f/d)^(1/3)
....: B=sqrt(b*A + g/A + h)
....: alpha = 1/a*B+ 1/2*sqrt(-A - i/j/A - k/b/B+ l/b) + m/n
sage: str(S[3].rhs().expand())==str(alpha)
True
\end{verbatim}
\subsection{Principal Graphs}\label{Sec: Principal Graph}
For the definition of fusion category and multifusion category, we refer the readers to  \cite{ENO05} and \cite[Section 1.12]{EGNO15}. 
The Grothendieck ring of a (multi)fusion category is a (multi)fusion ring.
A unitary (multi)fusion category is a (multi)fusion category with a unitary structure, see e.g. \cite[Remark 9.4.7]{EGNO15}.

 Let $\mathcal{N}\subseteq \mathcal{M}$ be a finite index subfactor.  Its subfactor planar algebra captures the bimodule category of the subfactor, $\{ _{\mathcal{N}}\mathsf {Mod}_\mathcal{N}$, $_{\mathcal{N}}\mathsf {Mod}_\mathcal{M}$, $_{\mathcal{M}}\mathsf {Mod}_\mathcal{M}$, $_{\mathcal{M}}\mathsf {Mod}_\mathcal{N} \}$.
This bimodule category is a unitary multitensor category with the $0$-morphisms being factors $\mathcal{N}$ and $\mathcal{M}$; the $1$-morphisms being bimodules over factors; and the $2$-morphisms being the bimodule homomorphisms. 
The subcategories $_{\mathcal{N}}\mathsf {Mod}_\mathcal{N}$ and 
$_{\mathcal{M}}\mathsf {Mod}_\mathcal{M}$ are unitary tensor categories. 
If they have finitely many simple objects, then the subfactor is called finite depth.
Take the generating bimodule $\tau=\sideset{_\mathcal{N}}{}{\mathop{\mathcal{M}}}_\mathcal{M}$.  Its conjugate $\overline{\tau}$ is the bimodule 
$\sideset{_\mathcal{M}}{}{\mathop{\mathcal{M}}}_\mathcal{N}$. Moreover, $\gamma=\tau\overline{\tau}=\sideset{_\mathcal{N}}{}{\mathop{\mathcal{M}}}_\mathcal{N}$. It defines a Frobenius algebra $(\gamma, m, \iota)$ in $_{\mathcal{N}}\mathsf {Mod}_\mathcal{N}$, where $m\in \hom(\gamma^2,\gamma)$ is the multiplication on $\mathcal{M}$ and $\iota$ is the inclusion from $\mathcal{N}$ to $\mathcal{M}$.
Moreover, $\tilde{\gamma}=\overline{\tau}\tau$ induces a Frobenius algebra in $_{\mathcal{M}}\mathsf {Mod}_\mathcal{M}$.
Conversely, given a Frobenius algebra in a unitary tensor category, one can construct a subfactor planar algebra as above, see e.g. \cite{Mug03}.

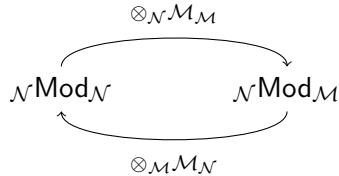
\begin{align*}
\raisebox{-0.85cm}{
\begin{tikzpicture}
\begin{scope}[shift={(0, 0.3)}]
\draw  (0, 0) .. controls +(0, 0.5) and +(0, 0.5) .. (3, 0)[->];
\end{scope}
\begin{scope}[shift={(0, -0.3)}]
\draw  (0, 0)[<-] .. controls +(0, -0.5) and +(0, -0.5) .. (3, 0);
\end{scope}
\node at (0,0) {\small $_{\mathcal{N}}\mathsf {Mod}_\mathcal{N}$};
\node at (3,0) {\small $_{\mathcal{N}}\mathsf {Mod}_\mathcal{M}$};
\node at (1.5,1) {\tiny $\otimes _{\mathcal{N}}\mathcal{M}_\mathcal{M}$};
\node at (1.5,-1) {\tiny $\otimes _{\mathcal{M}}\mathcal{M}_\mathcal{N}$};
\end{tikzpicture}}
\end{align*}

The principal graph of the subfactor $\mathcal{N}\subseteq \mathcal{M}$ is a bipartite graph. Its
vertices are equivalence classes of irreducible bimodules over $(\mathcal{N},\mathcal{N})$ and $(\mathcal{N},\mathcal{M})$. 
The number of edges connecting two vertices, an $(\mathcal{N},\mathcal{N})$
bimodule $Y$ and an $(\mathcal{N},\mathcal{M})$ bimodule $Z$, is the multiplicity of the equivalence class of $Z$
as a sub bimodule of $Y \otimes X$, where $X=\sideset{_\mathcal{N}}{}{\mathop{\mathcal{M}}}_\mathcal{M}$. The dimension vector of the bipartite graph is a function  from
the vertices of the graph to $\mathbb{R}^+$. Its value at a vertex is defined to be the dimension of
the corresponding bimodule. The dual principal graph is defined similarly for $(\mathcal{M},\mathcal{M})$ and $(\mathcal{M},\mathcal{N})$ bimodules. 
The adjacent matrix of the principal graph is a submatrix of the fusion matrix of $X$.


\subsection{Quantum Double Construction}\label{Sec: Quantum Double}
 Given a unitary fusion category, one can obtain a subfactor planar algebra from the quantum double construction. 
Here, we recall the construction with notations in \cite{LiuXu19}. Let $\mathscr{C}$ be a unitary fusion category and $\Irr=\{X_1=1, \ldots, X_m\}$ the set of all simple objects.
For each object $X_j$, we denote by $d_j:=\FPdim(X_j)$ the Frobenius-Perron dimension of $X_j$.
Denote by $\FPdim(\mathscr{C})$ the dimension of $\mathscr{C}$, which is $\displaystyle \sum_{j=1}^m \FPdim(X_j)^2$.
Let $\mathscr{C}^{op}$ be the opposite category of $\mathscr{C}$.
Then 
\begin{align*}
\gamma= \bigoplus_{X\in \Irr} X\boxtimes X^{op}
\end{align*}
is a Frobenius *-algebra in $\mathscr{C}\boxtimes \mathscr{C}^{op}$, where the multiplication $\Mu$ is given by
\begin{align*}
\Mu= \FPdim(\mathscr{C})^{1/4} \bigoplus_{j,k,t=1}^m (d_jd_kd_t)^{1/2} \sum_{\alpha\in ONB(X_j, X_k; X_t)} \alpha \boxtimes \overline{\alpha}
\end{align*}
and the canonical unit is $\iota: 1\boxtimes 1^{op} \to \gamma$; $ONB(X_j, X_k; X_t)$ is an orthonormal basis
of $\hom_{\mathscr{C}}(X_j\otimes X_k,X_t)$.
One can construct a subfactor planar algebra $\sP$ with a simple object $\tau$ associated to a single string such that $\gamma=\tau \overline{\tau}$.
The single quon space $\sP_{2,+}:=\hom_{\sP}(\tau\overline{\tau},\tau\overline{\tau})\cong\hom_{\mathscr{C}\boxtimes \mathscr{C}^{op}} (\gamma, \gamma)$ has an orthonormal basis 
$\{\beta_j\}_{j=1}^m$, where $\displaystyle \beta_{j}=d_j^{-1} \mathbbm{1}_{ X_j\boxtimes X_j^{op}}$, and $\mathbbm{1}_{ X_j\boxtimes X_j^{op}}$ is the identity map on $X_j\boxtimes X_j^{op}$. The morphism $\beta_j$ is represented as a square-like diagram in the planar algebra
\begin{align*}
\beta_j=\raisebox{-0.6cm}{
\begin{tikzpicture}
\begin{scope}[rotate=90]
\draw (-0.4, 0.3)--(0,0)--(0.5,0)--(0.9, 0.3) (-0.4, -0.3)--(0, 0) (0.9, -0.3)--(0.5, 0);
\end{scope}
\node at (0.1,0.25) {\tiny $j$};
\end{tikzpicture}}.
\end{align*}
The multiplication of the morphisms is a vertical composition and the convolution is a horizontal composition. Moreover, (Propositions 4.1, 4.2 in \cite{LiuXu19})
\begin{align*}
\beta_{j} \beta_{k}&=\delta_{j,k} d_{j}^{-1} \beta_{j} \;,\\
\beta_{j} *\beta_{k}&=\delta^{-1}\sum_{s=1}^m N_{j,k}^t \beta_{t} \;,
\end{align*}
where $\delta_{j,k}$ is the Kronecker delta function, $\displaystyle \delta^2=\sum_{j=1}^m d_j^2$ is the Jones index of the subfactor planar algebra, identical to the Frobenius-Perron dimension of $\mathscr{C}$.
Moreover, $\sP_{2,-}:=\hom_{\sP}( \overline{\tau}\tau, \overline{\tau}\tau)\cong\hom_{Z(\mathscr{C})} (\tilde{\gamma}, \tilde{\gamma})$, where $Z(\sC)$ is the Drinfeld center of $\sC$ and $\tilde{\gamma}=\overline{\tau}\tau$ induces a Frobenius algebra in $Z(\sC)$. Furthermore, $Z(\mathscr{C})$ and $\mathscr{C}\boxtimes \mathscr{C}^{op}$ are Morita equivalent w. r. t. the Frobenius algebras, see e.g. \cite{MugII03}. 

The Fourier transform $\fF_s: \sP_{2,\pm}\to \sP_{2,\mp}$ is a clockwise $90^{\circ}$ rotation, which intertwines the multiplication and the convolution.
Pictorially,
\begin{align*}
\fF_s(\beta_j)=\raisebox{-0.25cm}{
\begin{tikzpicture}
\draw (-0.4, 0.3)--(0,0)--(0.5,0)--(0.9, 0.3) (-0.4, -0.3)--(0, 0) (0.9, -0.3)--(0.5, 0);
\node at (0.25, 0.2) {\tiny $j$};
\end{tikzpicture}},
\end{align*}
and 
\begin{align}
\label{eq: 1}
\fF_s(\beta_{j}) * \fF_s(\beta_{k})&=\delta_{j,k} d_{j}^{-1} \fF_s(\beta_{j}) \;,\\
\label{eq: 2}
\fF_s^{-1}(\beta_{j}) \fF_s^{-1}(\beta_{k}) &=\delta^{-1}\sum_{t=1}^m N_{j,k}^t \fF_s^{-1}(\beta_{t}) \;.
\end{align}
By Equation \eqref{eq: 2}, $\sP_{2,-}$ is isomorphic to the $C^*$ algebra $\mathcal{R}\otimes_\mathbb{Z} \mathbb{C}$.

One can construct a subfactor planar algebra from a unitary multifusion categories following the quantum double construction in a similar way.

 \bibliographystyle{plain}

\begin{thebibliography}{99}
{

\bibitem{Ati88} M. Atiyah, {\sl Topological quantum field theories,} Publications Math\'{e}matiques de I'H\'{E}S, \textbf{68} (1988) 175-186.



\bibitem{BDLR19}  K. Bakshi, S. Das, Z. Liu, Y. Ren, {\sl An angle between intermediate subfactors and its rigidity,} Tran. Amer. Math. Soc., \textbf{371} (2019), 5973--5991.





\bibitem{BJ00} D. Bisch, V. Jones, {\sl Singly generated planar algebras of small dimension}, Duke Math. J., \textbf{101} (2000), 41--75.

\bibitem{BJ03} D. Bisch, V. Jones, {\sl Singly generated planar algebras of small dimension, Part II}, Adv. Math., \textbf{175} (2003), 297--318.

\bibitem{BJL17} D. Bisch, V. Jones,  Z. Liu, {\sl Singly generated planar algebras of small dimension, Part III}, Tran. Amer. Math. Soc., \textbf{369} (2017), 2461--2476.




\bibitem{EGNO15} P.~Etingof, S.~Gelaki, D.~Nikshych, V.~Ostrik, {\sl Tensor Categories}, American Mathematical Society, (2015). \newblock Mathematical Surveys and Monographs Volume 205.



\bibitem{ENO05}  P.~Etingof, D.~Nikshych, and V.~Ostrik, {\sl On fusion categories},
  Ann. Math., \textbf{162} (2005), 581--642.

  
  \bibitem{ENO21}P.~Etingof, D.~Nikshych, and V.~Ostrik, {\sl On a necessary condition for unitary categorification of fusion
rings}, arXiv:2102.13239, 2021.
 
 \bibitem{Haa94} U. Haagerup {\sl Principal graphs of subfactors in the index range $4< [M:N]< 3+\sqrt{2}$}, Subfactors (Kyuzeso, 1993), World Sci. Publ., River Edge, NJ, 1994. 
  
\bibitem{HLW21} L. Huang, Z. Liu, J. Wu, {\sl Quantum convolution inequalities on Frobenius von Neumann algebras}, arXiv:2204.04401, 2021.

\bibitem{Izu15} M. Izumi {\sl A Cuntz algebra approach to the classification of near group categories}, Proceedings of the Centre for Mathematics and its Applications, \textbf{46} (2015), 222--343. 

 \bibitem{IJMS12} M. Izumi, V. Jones, S. Morrison, N. Snyder, {\sl Subfactors of Index Less Than 5, Part 3: Quadruple Points.}  Commun. Math. Phys. \textbf{316}(2012), 531–554. 
 
 \bibitem{JJLRW20} A. Jaffe, C. Jiang, Z. Liu, Y. Ren, J. Wu, {\sl Quantum Fourier analysis}, Proc. Natl. Acad. Sci., \textbf{117} (2020), 10715--10720.

\bibitem{JMS14} V. Jones, S. Morrison, N. Snyder, {\sl The classification of subfactors of index at most 5}, B, \textbf{51} (2014), 277--327.

\bibitem{JLW16} C. Jiang, Z. Liu, J. Wu, {\sl Noncommutative uncertainty principles}, J. Funct. Anal., \textbf{270} (2016), 264--311.

 \bibitem{JLW18} C. Jiang, Z. Liu, J. Wu, {\sl Uncertainty principles for locally compact quantum groups}, J. Funct. Anal., \textbf{274} (2018), 2399--2445.

 \bibitem{JLW19} C. Jiang, Z. Liu, J. Wu, {\sl Block maps and Fourier analysis}, Sci. China Math., \textbf{62} (2019), 1585--1614.

 \bibitem{Jon83} V. Jones, {\sl Index for subfactors}, Invent. Math., \textbf{72} (1983), 1–25.
 
 \bibitem{Jon85} V. Jones, {\sl A polynomial invariant for knots via von Neumann algebras}, Bull. Amer. Math. Soc. , \textbf{12} (1985), 103–112.

\bibitem{Jon99} V. Jones, {\sl Planar algebras, I}, New Zealand J. Math., \textbf{52} (2021), 1--107.

\bibitem{Jon12} V. Jones, {\sl Quadratic tangles in planar algebras}, Duke Math J., \textbf{161} (2012), 2257--2295.


 
 \bibitem{Liu15} Z. Liu, {\sl Composed inclusions of $A_3$ and $A_4$ subfactors},  Adv. Math, \textbf{279} (2015), 307--371.

 \bibitem{Liu16} Z. Liu, {\sl Exchange relation planar algebras of small rank}, Trans. Amer. Math. Soc., \textbf{368} (2016), 8303--8348.

\bibitem{Liu19} Z. Liu {\sl Quon Language: Surface Algebras and Fourier Duality}, Commun. Math. Phys.,\textbf{366} (2019), 865--894.

\bibitem{LMP15} Z. Liu, S. Morrison, D. Penneys, {\sl 1-supertransitive subfactors with index at most 6+1/5}, Commun. Math. Phys., \textbf{334} (2015), 889--922.

\bibitem{LPW21} Z. Liu, S. Palcoux, J. Wu, {\sl Fusion bialgebras and Fourier Analysis:Analytic obstructions for unitary categorification}, Adv. Math., \textbf{390} (2021), 107905.

\bibitem{LPR22} Z. Liu, S. Palcoux, Y. Ren, {\sl Classification of Grothendieck rings of complex fusion categories of multiplicity one up to rank six}, Lett. Math. Phys. \textbf{112}, \textbf{54} (2022).

\bibitem{Lus87} {\sc G.~Lusztig}, {\sl Leading coefficients of character values of Hecke algebras}, Proc. Symp. in Pure Math., \textbf{47} (1987), 235--262.

\bibitem{LW18} Z. Liu, J.Wu, {\sl Extremal pairs of Young’s inequality for Kac algebras}, Pacific J. Math., \textbf{295} (2018), 103--121.

\bibitem{LiuXu19} Z. Liu, F. Xu, {\sl Jones-Wassermann subfactors for modular tensor categories}, Adv. Math, \textbf{335} (2019), 106775.


\bibitem{MPPS12} S. Morrison, D. Penneys, E. Peters, and N. Snyder, {\sl Subfactors of index less than 5, Part 2: Triple points}, Internat. J. Math., \textbf{23} (2012), 1250016.

\bibitem{Mug03}
M. M\"uger, {\sl From subfactors to categories and topology. I. Frobenius algebras in and Morita
equivalence of tensor categories}, J. Pure Appl. Algebra, \textbf{180} (2003), 81--157.

\bibitem{MugII03}
M. M\"uger, 
{\sl From subfactors to categories and topology II: The quantum double of tensor categories and subfactors}, 
J. Pure Appl. Algebra,  \textbf{180} (2003), 159--219.

\bibitem{Ocn94} A.~Ocneanu, {\sl  Chirality for operator algebras}, In: H. Araki, Y. Kawahigashi and H. Kosaki (eds): Subfactors. 39--63. World Scientific Publ., 1994.

\bibitem{Ost033} V.~Ostrik, {\sl  Fusion categories of rank 2}, Math. Res. Lett., \textbf{10} (2003), 177--183.


\bibitem{Ost15} V.~Ostrik, {\sl Pivotal fusion categories of rank 3}, Mosc. Math. J., \textbf{15} (2015), 373--396, 405.

\bibitem{Pen15} D.~Penneys, {\sl Chirality and principal graph obstructions}, Adv. Math., \textbf{273} (2015), 32--55.

\bibitem{Pen20} D.~Penneys, {\sl Unitary dual functors for unitary multitensor categories}, Higher Structures, \textbf{4} (2020), 
22--56.



\bibitem{ResTur91} N. Reshetikhin, V. Turaev, {\sl Invariants of 3-manifolds via link polynomials and quantum groups}, Invent. Math., \textbf{103} (1991), 547--597.

\bibitem{Sny13} N. Snyder, {\sl  A rotational approach to triple point obstructions}, Analysis \& PDE, \textbf{6} (2013), 1923-1928.




\bibitem{TurVir92} V. Turaev and O. Viro {\sl State sum invariants of 3-manifolds and quantum 6j-symbols.} Topology, \textbf{31} (1992), 865--902.

\bibitem{VS22} G. Vercleyen, J. K. Slingerland, {\sl On low rank fusion rings}, arXiv:2205.15637, 2022.



\bibitem{Wit88} E. Witten, {\sl  Topological quantum field theory}, Commun. Math. Phys, \textbf{117} (1988), 353--386.

\bibitem{You12} W. Young, {\sl On the multiplication of successions of Fourier constants}, Proc. Roy. Soc. London, Ser. A, \textbf{87} (1912), 331--339.

}

\end{thebibliography}

\end{document}